\documentclass[12pt,reqno]{amsart}

\usepackage[leqno]{amsmath}
\usepackage{amsthm}
\usepackage{amsfonts, tikz-cd}
\usepackage{amssymb}
\usepackage{eucal}
\usepackage[all]{xy}
\usepackage{mathtools}
\usepackage{stmaryrd}
\usepackage{csquotes}
\usepackage{enumitem}
\usepackage{upgreek}
\usepackage{color}

\setenumerate{itemsep=2pt,parsep=2pt,before={\parskip=2pt}}

\usepackage[colorlinks=true,hyperindex, linkcolor=magenta, pagebackref=false, citecolor=cyan,pdfpagelabels]{hyperref}

\usepackage{relsize}
\usepackage[bbgreekl]{mathbbol}
\DeclareSymbolFontAlphabet{\mathbb}{AMSb} 
\DeclareSymbolFontAlphabet{\mathbbl}{bbold}

\usepackage{xspace}
\usepackage{epsfig,epic,eepic,latexsym,color}
\usepackage[all]{xy}
\usepackage{comment}

\usepackage[english]{babel}
\usepackage{latexsym}

\newcommand{\nc}{\newcommand}
\nc{\rnc}{\renewcommand}
\rnc{\P}{\mathbf P}
\nc{\R}{\mathbf R}
\rnc{\rm}{\mathrm}

\nc{\C}{\mathbf C}
\nc{\Q}{\mathbf Q}
\nc{\Z}{\mathbf Z}
\nc{\N}{\mathbf N}
\nc{\A}{\mathbf A}

\nc{\cH}{\mathcal{H}}
\nc{\rR}{\mathrm{R}}

\nc{\an}{\operatorname{an}}
\nc{\red}{\operatorname{red}}
\nc{\coker}{\operatorname{coker}}

\nc{\et}{\text{\'et}}
\nc{\fet}{\text{f\'et}}

\nc{\htt}{\operatorname{ht}}
\nc{\Nm}{\operatorname{Nm}}
\nc{\Ker}{\operatorname{Ker}}
\nc{\mmod}{\operatorname{mod}}
\nc{\End}{\operatorname{End}}
\nc{\Aut}{\operatorname{Aut}}
\nc{\cont}{\text{cont}}
\nc{\sep}{\text{sep}}
\nc{\Hom}{\mathrm{Hom}}
\nc{\Gal}{\mathrm{Gal}}
\nc{\Spec}{\text{Spec}\,}
\nc{\RZ}{\operatorname{RZ}}
\nc{\HHom}{\ud{\rm{Hom}}}
\nc{\hocolim}{\rm{hocolim}}
\nc{\diam}{\diamondsuit}
\nc{\cl}{\rm{cl}}

\rnc{\t}{\tau}
\nc{\mm}{\pmb{\mu}}
\rnc{\a}{\alpha}
\nc{\n}{\mathfrak n}
\nc{\m}{\mathfrak m}
\nc{\mfs}{\mathfrak s}
\nc{\Cat}{\cal{C}\rm{at}}
\rnc{\Pr}{\cal{P}\rm{r}^{\rm{L}}}

\nc{\p}{\mathfrak p}
\nc{\q}{\mathfrak q}

\nc{\Sym}{\operatorname{Sym}}
\nc{\codim}{\operatorname{codim}}
\nc{\rk}{\operatorname{rk}}
\nc{\GL}{\operatorname{GL}}
\nc{\SL}{\operatorname{SL}}
\nc{\Lie}{\operatorname{Lie}}
\nc{\Ind}{\operatorname{Ind}}
\nc{\Div}{\underline{Div}}
\nc{\Pic}{\mathbf{Pic}}
\nc{\uPic}{\underline{ \mathbf{Pic}}}
\nc{\rH}{\mathrm{H}}
\nc{\Spf}{\operatorname{Spf}}
\nc{\Frac}{\operatorname{Frac}}
\nc{\colim}{\operatorname{colim}}
\nc{\Spa}{\operatorname{Spa}}
\nc{\tr}{\operatorname{tr}}
\nc{\Corr}{\operatorname{Corr}}
\nc{\coh}{\operatorname{coh}}
\nc{\Coh}{\operatorname{Coh}}
\nc{\rPic}{\mathrm{Pic}}
\nc{\alg}{\mathrm{alg}}

\rnc{\an}{\operatorname{an}}

\nc{\xr}{\xrightarrow}

\nc{\eps}{\epsilon}
\nc{\ov}{\overline}
\nc{\ud}{\underline}
\nc{\wdh}{\widehat}

\nc{\I}{\mathcal I}
\nc{\F}{\mathcal F}
\nc{\G}{\mathcal G}
\nc{\E}{\mathcal E}
\nc{\M}{\mathcal M}
\nc{\cal}{\mathcal}
\rnc{\rm}{\mathrm}
\rnc{\bf}{\mathbf}

\nc{\X}{\mathfrak X}
\nc{\Y}{\mathfrak Y}
\nc{\T}{\mathfrak T}

\nc{\LL}{\mathcal{L}}
\rnc{\S}{\mathcal S}

\nc{\ra}{\rangle}

\nc{\os}{\overset}

\rnc{\O}{\mathcal O}
\nc{\J}{\mathcal J}
\theoremstyle{definition}

\newtheorem{thm1}{Theorem}[section]
\newtheorem{lemma1}[thm1]{Lemma}

\newtheorem{prop1}[thm1]{Proposition}

\newtheorem{rmk1}[thm1]{Remark}

\newtheorem{thm}{Theorem}[section]
\newtheorem{lemma}[thm]{Lemma}
\newtheorem{defn}[thm]{Definition}

\newtheorem{notation}[thm]{Notation}

\newtheorem{warning}[thm]{Warning}

\newtheorem{rmk}[thm]{Remark}

\newtheorem{exmpl}[thm]{Example}
\newtheorem{cor}[thm]{Corollary}

\begin{document}
\bibliographystyle{halpha-abbrv}
\title{Some foundational results in adic geometry}
\author{Bogdan Zavyalov}
\maketitle

\begin{abstract}
In this paper, we record some foundational results on adic geometry that seem to be missing in the existing literature. Namely, we develop the Proj construction and a theory of lci closed immersions in the context of locally noetherian analytic adic spaces. In the context of rigid-analytic spaces, these topics have previously been considered in \cite{conrad-ample} and \cite{Guo-Li} respectively. We also develop an \'etale six functor formalism in the analytic geometry and give a categorical description of lisse and constructible sheaves. All results of this paper are probably well-known to the experts. 
\end{abstract}

\tableofcontents
\section{Introduction} 

The main goal of this paper is to record certain foundational results about locally noetherian analytic adic spaces that seem to be missing in the existing literature. The primary motivation comes from our companion paper \cite{duality-revisited}, where we use many of the results of this paper to give a ``formal'' proof of Poincar\'e Duality in non-archimedean (and algebraic) geometry. However, we hope that these results could be useful for other people working in the area of general (noetherian) adic spaces. \smallskip

We now mention the main concepts discussed in this paper. First, we develop a theory of lci immersions for general locally noetherian analytic adic spaces. This theory has been worked out for rigid-analytic spaces in \cite[Appendix]{Guo-Li}; however, the notion of lci (immersions) on more general analytic adic spaces seems to be missing in the literature. We refer to Definition~\ref{defn:lci-immersion} for the precise definition of lci immersions and only summarize its main properties below:

\begin{thm}[Lemma~\ref{lemma:flat-base-change}, Corollary~\ref{cor:smooth-smooth-lci}, Corollary~\ref{cor:lci-relative-analytification}]
\begin{enumerate}
    \item Lci immersions of locally noetherian analytic adic spaces are closed under flat pullbacks;
    \item A section of a smooth, separated morphism is an lci immersion;
    \item A (relative) analytification of an lci immersion is an lci immersion.
\end{enumerate}
\end{thm}

In order to achieve these results, we also recall some facts about coherent sheaves and relative analytifications in Sections~\ref{section:sheaves} and \ref{section:analytification}, respectively. \smallskip

The next big topic that we discuss is the relative analytic Proj construction. In particular, we provide the construction of Proj of a graded locally coherent $\O_S$-algebra (see Definition~\ref{defn:locally-coherent}) on any locally noetherian analytic adic space $S$. In particular, this allows us to define blow-ups and projective vector bundles in big generality; see Definition~\ref{defn:projective-bundle} and Definition~\ref{defn:blow-ups}. This theory has been worked out for rigid-analytic spaces in \cite{conrad-ample} by a different approach; however, the case of more general analytic adic spaces seems to be missing in the literature.

We refer to Definition~\ref{defn:proj-general} for the precise definition of our Proj construction and only (informally) summarize its main properties below:

\begin{thm}[Def.~\ref{defn:proj-general}, Rmk.~\ref{rmk:proj-base-change}, Def.~\ref{defn:projective-bundle}, Def.~\ref{defn:blow-ups} Lemma~\ref{lemma:proj-commutes-with-analytification}, Lemma~\ref{lemma:blow-up-analytification}, Thm.~\ref{thm:relative-picard-projective-line}]\label{thm:intro-2}
\begin{enumerate}
    \item There is a good notion of relative analytic Proj construction $\ud{\rm{Proj}}^{\an}_S\cal{A}_\bullet$ which commutes with arbitrary pullbacks and (relative) analytifications;
    \item There is a good notion of blow-ups and projectivized vector bundles on locally noetherian analytic adic spaces. Moreover, these notions commute with (relative) analytifications;
    \item\label{thm:intro-2-3} For any locally noetherian connected analytic adic space $S$, a vector bundle $\cal{E}$ on $S$ of rank at least $2$, and the corresponding projective bundle $P\coloneqq \bf{P}_S(\cal{E}) \to S$, the natural morphism
    \[
        \rm{Pic}\big(S\big) \bigoplus \Z \to \rm{Pic}\big(P\big)
    \]
    \[
        \cal{L} \oplus n \mapsto f^*\cal{L} \otimes_{\O_P} \O_{P/S}(n)
    \]
    is an isomorphism. 
\end{enumerate}
\end{thm}

In order to prove Theorem~\ref{thm:intro-2}(\ref{thm:intro-2-3}), we need to establish certain basic facts about connected components of locally noetherian analytic adic spaces. We record all these facts in Section~\ref{section:components}. \smallskip

The next big topic of this paper is the construction of \'etale six functor formalism in the sense of \cite[Def.~2.3.10]{duality-revisited} (crucially based on \cite[Appendix A.5]{Lucas-thesis}) that extends the six functors constructed in \cite{H3}. \smallskip

\begin{thm}[Theorem~\ref{thm:etale-six-functors}]\label{thm:intro-3} Let $S$ be a locally noetherian analytic adic space, let $\cal{C}'$ be the category of locally $+$-weakly finite type $S$-adic spaces, and let $n>0$ be an integer {\it invertible} in $\O_S^+$. Then there is a $6$-functor formalism 
\[
\cal{D}_{\et}(-; \Z/n\Z) \colon \rm{Corr}(\cal{C'}) \to \Cat_\infty
\]
such that 
\begin{enumerate}
    \item there is a canonical isomorphism of symmetric monoidal $\infty$-categories $\cal{D}_\et(X;\Z/n\Z) = \cal{D}(X_\et; \Z/n\Z)$ for any $X\in \cal{C'}$;
    \item for a  morphism $f\colon X \to Y$ in $\cal{C}'$, we have 
    \[
    \cal{D}_\et\big([Y \xleftarrow{f} X \xr{\rm{id}} X]\big) = f^*\colon \cal{D}(Y_\et; \Z/n\Z) \to \cal{D}(X_\et; \Z/n\Z);
    \]
    \item for a separated, taut, locally $+$-weakly finite type morphism $f\colon X \to Y$, we have
    \[
    \cal{D}_{\et}\big([X \xleftarrow{\rm{id}} X \xr{f} Y]\big)|_{\cal{D}_\et^+(X_\et; \Z/n\Z)} \simeq \rm{R}^+f_! \colon \cal{D}^+_\et(X_\et; \Z/n\Z) \to \cal{D}_\et(Y_\et; \Z/n\Z),
    \]
    where $\rm{R}^+f_!$ is the functor from \cite[Thm.~5.4.3]{H3}.
\end{enumerate}
\end{thm}

We note that Theorem~\ref{thm:intro-3} constructs the $\rm{R}f_!$-functor on the unbounded derived categories for an arbitrary locally $+$-weakly finite type morphism of locally noetherian analytic adic spaces. In particular, no separatedness, tautness, or boundedness assumptions are necessary. \smallskip

We note that \cite[Appendix A.5]{Lucas-thesis} provides exceptionally convenient and useful machinery for the purpose of constructing six functor formalisms. Using {\it loc.~cit.} the main content of Theorem~\ref{thm:intro-3} essentially boils down to verifying the unbounded version of proper base change and projection formula. In order to do this, we need to recall some facts from the dimension theory on adic spaces. We do this in Section~\ref{section:dimension}. \smallskip

Finally, we provide categorical descriptions of (the derived category of) lisse and constructible sheaves on noetherian analytic adic spaces and schemes.

\begin{thm}[Lemma~\ref{lemma:lisse-categorical}, Lemma~\ref{lemma:abstract-nonsense}]\label{thm:intro-4} Let $X$ be a locally noetherian analytic adic space or a scheme, and let $n>0$ be an integer.
    \begin{enumerate}
        \item The following are equivalent:
        \begin{itemize}
            \item $\F\in \cal{D}(X_\et; \Z/n\Z)$ is dualizable;
            \item $\F\in \cal{D}(X_\et; \Z/n\Z)$ is perfect;
            \item $\F$ lies in $\cal{D}^{(b)}_{\rm{lisse}}(X_\et; \Z/n\Z)$ and, for each geometric point $\ov{s} \to X$, the stalk $\F_{\ov{s}}$ is a perfect complex in $\cal{D}(\Z/n\Z)$.
        \end{itemize}
        \item Assume that $X$ is also quasi-compact and quasi-separated, and let $N$ be an integer. Then an object $\F\in \cal{D}^{\geq -N}(X_\et; \Z/n\Z)$ is compact if and only if $\F$ lies in $\cal{D}^{b, \geq -N}_{\rm{cons}}(X_\et; \Z/n\Z)$, i.e., $\F$ is bounded and all its cohomology sheaves are constructible.
    \end{enumerate}
\end{thm}

Theorem~\ref{thm:intro-4} is a very handy tool in getting certain finiteness statements from Poincar\'e Duality for free (see \cite[Application 1.3.4]{duality-revisited}).

\subsection{Terminology}

We say that an analytic adic space $X$ is {\it locally noetherian} if there is an open covering by affinoids $X= \bigcup_{i\in I} \Spa(A_i, A_i^+)$ with strongly noetherian Tate $A_i$. Sometimes, such spaces are called locally {\it strongly} noetherian.\smallskip

We follow \cite[Def.\,1.3.3]{H3} for the definition of a locally finite type, locally weakly finite type, and locally $+$-weakly finite type morphisms of locally noetherian analytic adic spaces.  

For a Grothendieck abelian category $\cal{A}$, we denote by $D(\cal{A})$ its \emph{triangulated derived category} and by $\cal{D}(\cal{A})$ its $\infty$-enhancement. \smallskip

For a locally noetherian analytic adic space $X$ and an integer $n>0$, we denote by $\cal{D}^{(b)}(X_\et; \Z/n\Z)$ the full subcategory of $\cal{D}(X_\et; \Z/n\Z)$ consisting of {\it locally} bounded complexes. 

\subsection{Acknowledgements}

The author thanks the anonymous referee and Sean Howe for their valuable comments on the previous draft of this paper. The author is also grateful to the Institute for Advanced Study and Princeton University for funding and excellent working conditions.

\section{Connected Components}\label{section:components}

In this section, we study connected components of locally noetherian analytic adic spaces.

\begin{lemma}\label{lemma:connected-different} Let $X=\Spa(A, A^+)$ be a strongly noetherian Tate affinoid. Then $X$ is connected if and only if $\Spec A$ is connected.
\end{lemma}
\begin{proof}
    Both connectivity of $\Spa(A, A^+)$ and of $\Spec A$ are equivalent to the fact that $A$ does not admit any non-trivial idempotents\footnote{Here, we crucially use that $\Spa(A, A^+)$ is sheafy.}. In particular, they are equivalent to each other.
\end{proof}

\begin{lemma}\label{lemma:locally-connected} Let $X$ be a locally noetherian analytic adic space. Then any point $x\in X$ admits a fundamental system of connected affinoid open neighborhoods. In particular, $X$ is locally connected. 
\end{lemma}
\begin{proof}
    It suffices to show that, for any strongly noetherian Tate affinoid $X=\Spa(A, A^+)$ and a point $x\in X$, the connected component of $x$ is clopen. For this, note that the ring $A$ is noetherian, and so admits only a finite number of mutually orthogonal non-trivial idempotents. Therefore, $X$ has only a finite number of connected components, thus they all must be open and closed. 
\end{proof}

\begin{cor}\label{cor:finite-number-of-components} Let $X$ be a locally noetherian analytic adic space. Then each connected component of $X$ is closed and open. In particular, if $X$ is a noetherian analytic adic space, then $\pi_0(X)$ is finite. 
\end{cor}
\begin{proof}
    Connected components are always closed (see \cite[\href{https://stacks.math.columbia.edu/tag/004T}{Tag 004T}]{stacks-project}), so it suffices to show that they are open. This follows from  \cite[\href{https://stacks.math.columbia.edu/tag/04ME}{Tag 04ME}]{stacks-project} and Lemma~\ref{lemma:locally-connected}. \smallskip
\end{proof}

\section{Dimension}\label{section:dimension}

In this section, we study different possible definitions of dimension in non-archimedean geometry.  \smallskip

\begin{defn}\label{defn:dimension}(\cite[Def.\,1.8.1]{H3}) The {\it dimension} of a locally spectral $X$ is the supremum of the length $d$ of the chains of specializations $x_0\succ x_1 \succ \dots \succ x_d$ of points of $X$.

A locally spectral space $X$ is of {\it pure dimension $d$} if every non-empty open subset $U\subset X$ has dimension $d$.

The {\it (relative) dimension} $\dim f$ of a morphism of analytic adic spaces $f\colon X \to Y$ is the supremum of the dimensions of the fibers of $f$,
\[
\dim f \coloneqq \sup_{y\in Y} \dim f^{-1}(y)\in \bf{Z}_{\geq 0} \cup \{\infty\}.
\]
A morphism $f\colon X \to Y$ is of {\it relative pure dimension $d$} if all non-empty fibers $f^{-1}(y)$ are of pure dimension $d$.
\end{defn}

Firstly, it turns out that one can only consider fibers over rank-$1$ points. 

\begin{lemma}\label{lemma:only-rank-1}(\cite[Cor.\,1.8.7]{H3}) Let $S$ be a locally noetherian analytic adic space, and let $f\colon X \to S$ be a locally finite type morphism. Then $f$ is of relative pure dimension $d$ if and only if, for each rank-$1$ point $s\in S$, the fiber
\[
X_s\coloneqq X\times_S \Spa\left(K(s), \O_{K(s)}\right)
\]
is either empty or of pure dimension $d$.
\end{lemma}

It turns out that, in case of rigid-analytic spaces,  Definition~\ref{defn:dimension} recovers the usual notion of dimension:

\begin{lemma}\label{lemma:classical-dimension} Let $K$ be a non-archimedean field, and let $X$ be a rigid-analytic $K$-space. Then $X$ is of pure dimension $d$ if and only if, for each classical point $x\in X$, $\dim \O_{X, x}=d$.
\end{lemma}
\begin{proof}
    First, we note that \cite[Lemma 1.8.6(ii)]{H3} implies that $X$ is of pure dimension $d$ if and only if, for every open affinoid subspace $\Spa(A, A^\circ) \subset X$, we have $\dim A=d$. Then \cite[Prop.\,II.10.1.9 and Cor.\,II.10.1.10]{FujKato} imply that this condition is equivalent to the condition that $\dim \O_{X, x}=d$ for any classical point $x\in X$. 
\end{proof}

\begin{cor}\label{cor:dimension-of-disc} Let $S$ be a locally noetherian analytic adic space, and let $f\colon X \to S$ be a morphism that factors as a composition
\[
X \xr{g} \bf{D}^d_S \xr{\pi} S,
\]
where $g$ is \'etale and $p$ is the natural projection. Then $f$ is of pure relative dimension $d$.
\end{cor}
\begin{proof}
    By Lemma~\ref{lemma:only-rank-1}, it suffices to assume that $S=\Spa(K, \O_K)$ for some non-archimedean field $K$. In this situation, the question reduces to showing that $X$ is an adic space of pure dimension $d$. Then \cite[Example 1.8.9(i)]{H3} does the job (alternatively, one can use that \'etale morphisms are of pure dimension $0$ and \cite[Proposition 2.2/17 and Proposition 4.1/2]{B} to get the claim).
\end{proof}

Now, we wish to show that any weakly finite type morphism $f$ is of finite (relative) dimension. Surprisingly, this claim seems missing in \cite{H3}. For this, we need a number of preliminary lemmas that will allow us to reduce the general case to the case when $f$ is of finite type. The motivation for considering non-finite type morphisms comes from the theory of universal compactifications that are (essentially) never finite type (and merely $+$-weakly finite type).

\begin{lemma}\label{lemma:weakly-finite-almost-finite} Let $(A, A^+) \to (B, B^+)$ be a morphism of strongly noetherian Tate pairs such that $B$ is finite over $A$. Denote by $B'^+$ the integral closure of $A^+$ in $B$. Then
\begin{enumerate}
    \item $(B, B'^+)$ is a Huber pair;
    \item $(A, A^+) \to (B, B'^+)$ is a finite morphism.
\end{enumerate}
\end{lemma}
\begin{proof}
    The subring $B'^+$ of $B$ is clearly integrally closed in $B$. It is also contained in $B^\circ$ because $B'^+ \subset B^+ \subset B^\circ$. So, in order to show that $(B, B'^+)$ is a Huber pair, we only need to show that it is open.  \smallskip
    
    For this, we choose a ring of definition $A_0$, a pseudo-uniformizer $\varpi \in A_0$, and $(b_1, \dots, b_n)$ a finite set of $A$-module generators of $B$.  Since $B$ is finite over $A$, for each generator $b_i\in B$, we can choose monic polynomials
    \begin{equation}\label{eqn:integral}
        b_i^{m_i} + a_{i, 1}b^{m_i-1}_{i} + \dots + a_{i, m_i} = 0
    \end{equation}
    with $a_{i, j}\in A$. By construction, there is an integer $N$ such that $\varpi^N a_{i, j}\in A_0$ for all $i,j$. Using Equation~(\ref{eqn:integral}), it is easy to see that all elements $\varpi^N b_i$ are {\it integral over $A_0$}. Thus, we can replace each $b_i$ with $\varpi^N b_i$ to assume that the $A$-module generators $b_i$ are integral over $A_0$. In particular, we can assume that each $b_i$ is integral over $A^+$, so they all lie in $B'^+ \subset B^\circ$. 
    \smallskip
    
    Now consider the unique $A$-linear morphism
    \[
    \varphi \colon A\langle T_1, \dots, T_n\rangle \to B
    \]
    that sends $T_i$ to $b_i$. It is clearly surjective, and therefore it is open by the Open Mapping theorem (see \cite[Lemma 2.4(i)]{H2}), so we define 
    \[
        B_0 \coloneqq \varphi(A_0\langle T_1, \dots, T_n\rangle).
    \]
    This is then a ring of definition in $B$ with a pseudo-uniformizer given by $\varpi$. By construction, the morphism 
    \[
    A_0/\varpi A_0 \to B_0/\varpi B_0
    \]
    is finite. Therefore, using that $A_0$ and $B_0$ are complete, we conclude that $B_0$ is finite over $A_0$. In particular, elements of $B_0$ are integral over $A^+$, so $B_0 \subset B'^+$. This ensures that $B'^+$ is open. This finishes the proof that $(B, B'^+)$ is a Huber pair. \smallskip
    
    The morphism $(A, A^+) \to (B, B'^+)$ is now clearly finite. Indeed, $A \to B$ is finite by the assumption, and $A^+ \to B'^+$ is integral by construction. 
\end{proof}

\begin{cor}\label{cor:change-plus-rings} Let $f\colon (A, A^+) \to (B, B^\circ)$ be weakly finite type morphism of strongly noetherian Tate affinoids. Then there is a Huber pair $(B, B^+)$ such that $f$ factors through $(B, B^+)\to (B, B^\circ)$, and $(B, B^+)$ is topologically finite type over $(A, A^+)$.
\end{cor}
\begin{proof}
    Since $(B, B^\circ)$ is weakly finite type over $(A, A^+)$, there is a surjective morphism
    \[
    g\colon A\langle T_1, \dots, T_n\rangle \to B.
    \]
    Since any morphism of Tate rings is adic, and adic morphisms preserve bounded elements (see \cite[Lemma 1.8]{H1}), we conclude $g$ induces a morphism of Huber pairs
    \[
    g\colon (A\langle T_1, \dots, T_n\rangle, A\langle T_1, \dots, T_n\rangle^+)  \to (B, B^\circ)
    \]
    Thus we can apply Lemma~\ref{lemma:weakly-finite-almost-finite} to $g$ to get a Huber sub-pair $(B, B^+)\subset (B, B^\circ)$ such that $B^+$ is integral over the image $A\langle T_1, \dots, T_n\rangle^+$ in $B$. In particular, the morphism $(A, A^+) \to (B, B^+)$ is topologically of finite type. 
\end{proof}

\begin{lemma}\label{lemma:+-weakly-finite-type-finite-dimensional} Let $f\colon X \to Y$ be a weakly finite type morphism, and let $Y$ be quasi-compact. Then $f$ is of finite dimension. 
\end{lemma}
\begin{proof}
    An easy argument with quasi-compactness reduces the general case to the case of a weakly finite type morphism of affinoid spaces $f\colon X=\Spa(B, B^+) \to Y=\Spa(A, A^+)$, i.e., $B$ is topologically of finite type over $A$. Now note that the natural inclusion 
    \[
    X'=\Spa(B, B^\circ) \to X=\Spa(B, B^+)
    \]
    is a bijection on rank-$1$ points. Therefore, $\rm{dim.tr}(X/Y)=\rm{dim.tr}(X'/Y)$ (see \cite[Def.\,1.8.4]{H3}). In particular, we can replace $B^+$ with $B^\circ$. \smallskip
    
    In this case, we apply Corollary~\ref{cor:change-plus-rings} and a similar argument once again to reduce to the case of a finite type morphism $\Spa(B, B^+)\to \Spa(A, A^+)$. In this case, there is  closed immersion
    \[
    X \to \bf{D}^n_Y,
    \]
    so it suffices to show the claim for the relative closed unit disk $\bf{D}^n_Y\to Y$. This case follows from Corollary~\ref{cor:dimension-of-disc}. 
\end{proof}

\section{Coherent sheaves}\label{section:sheaves}

In this section, we review the basic theory of coherent sheaves on locally noetherian analytic adic spaces. \smallskip

We first recall the construction of an $\O_X$-module $\widetilde{M}$ on a strongly noetherian analytic affinoid $X=\Spa(A, A^+)$ associated to a finite $A$-module $M$. For each rational subset $U \subset X$, we have 
\[
\widetilde{M}(U)=\O_X(U)\otimes_{A} M;
\]
\cite[Thm.\,1.4.16]{KedAr} and \cite[Thm.\,1.2.11]{KedAr} guarantee that this assignment is indeed a sheaf. 

\begin{defn} An $\O_X$-module $\F$ on a locally strongly noetherian analytic adic space $X$ is {\it coherent} if there is an open covering $X=\cup_{i\in I} U_i$ by strongly noetherian affinoids such that $\F|_{U_i} \cong \widetilde{M}_i$ for a finite $\O_X(U_i)$-module $M_i$.
\end{defn}

For a strongly noetherian analytic affinoid $X=\Spa(A, A^+)$, we get a functor $\widetilde{(-)}\colon \bf{Mod}_A^{\rm{fg}} \to \bf{Coh}_X$. Similarly to the algebraic situation, this functor turns out to be an equivalence. 

\begin{thm}\label{thm:huber-coherent-sheaves} Let $X=\Spa(A, A^+)$ be a strongly noetherian affinoid, and let $\F$ be a coherent $\O_X$-module. Then
\begin{enumerate}
    \item\label{thm:huber-coherent-sheaves-1} the functor $\widetilde{(-)}\colon \bf{Mod}_A^{\rm{fg}} \to \bf{Mod}_X$ is exact;
    \item\label{thm:huber-coherent-sheaves-2} the functor $\widetilde{(-)}\colon \bf{Mod}_A^{\rm{fg}} \to \bf{Coh}_X$ is an equivalence with quasi-inverse taking $\F$ to $\Gamma(X, \F)$;
    \item\label{thm:huber-coherent-sheaves-3} for any $\F\in \bf{Coh}_X$, $\rm{H}^i(X, \F)=0$ for $i\geq 1$.
    \item\label{thm:huber-coherent-sheaves-4} the inclusion $\bf{Coh}_X$ is a weak Serre subcategory of $\bf{Mod}_X$. In other words, coherent sheaves are closed under kernels, cokernels, and extensions;
\end{enumerate}
\end{thm}
\begin{proof}
    First note that $A$ is sheafy by \cite[Thm.\,1.2.11]{KedAr} or \cite[Cor.\,1.3]{sheafy}. So (\ref{thm:huber-coherent-sheaves-1}) can be easily deduced from \cite[Thm.\,1.4.14]{KedAr}.  (\ref{thm:huber-coherent-sheaves-2}) and (\ref{thm:huber-coherent-sheaves-3}) follow from \cite[Thm.\,1.4.18]{KedAr}. Finally, (\ref{thm:huber-coherent-sheaves-4}) can be deduced from all (\ref{thm:huber-coherent-sheaves-1}-\ref{thm:huber-coherent-sheaves-3}) by a standard argument. 
\end{proof}

\begin{lemma}\label{lemma:pullback-coherent} Let $f\colon X \to Y$ be a morphism of locally noetherian analytic adic spaces, and let $\F$ be a coherent $\O_Y$-module. Then
\begin{enumerate}
    \item\label{lemma:pullback-coherent-1} the pullback $f^*\F$ is a coherent $\O_X$-module;
    \item\label{lemma:pullback-coherent-2} if $X=\Spa(B, B^+)$ and $Y=\Spa(A, A^+)$ are affinoid and $\F=\widetilde{M}$ for a finite $A$-module $M$, then $f^*\F \simeq \widetilde{M\otimes_A B}$.
\end{enumerate}
\end{lemma}
\begin{proof}
    Clearly, (\ref{lemma:pullback-coherent-1}) follows from (\ref{lemma:pullback-coherent-2}). To prove (\ref{lemma:pullback-coherent-2}), we use noetherianness of $A$ to find a partial resolution
    \begin{equation}\label{eqn:partial-resolution}
    A^n \to A^m \to M \to 0.
    \end{equation}
    The claim is clear when $M=A^n$, so the general case follows from (\ref{eqn:partial-resolution}), exactness of $\widetilde{(-)}$, and right exactness of $f^*$.
\end{proof}

\section{Regular closed immersions}\label{section:immersion}

In this section, we first recall the notion of  Zariski-closed subspaces of locally noetherian analytic adic spaces. Then we discuss the theory of lci subspaces and, in particular, effective Cartier divisors. For rigid-analytic spaces, (a more general) theory of lci morphisms is developed in \cite{Guo-Li}.

\begin{defn}\label{defn:zariski-closed} A morphism $i\colon X \to Y$ of locally noetherian analytic adic spaces is a {\it Zariski-closed immersion} if $i$ is a homeomorphism of $X$ onto a closed subset of $Y$, the map $\O_Y \to i_*\O_X$ is surjective, and the kernel $\mathcal I\coloneqq \ker (\O_Y \to i_*\O_X)$ is coherent.
\end{defn}

We refer to \cite[Appendix B.6]{quotients} for a detailed discussion of this notion (studied under the name of closed immersions). In particular, we point out \cite[Cor.\,B.6.9]{quotients} that guarantees that a Zariski-closed subspace of a strongly noetherian Tate affinoid $X=\Spa(A, A^+)$ is a (strongly noetherian Tate) affinoid. Furthermore, Zariski-closed subspaces of $X$ are parametrized by the ideals $I\subset A$. \smallskip

Before we discuss the notion of lci immersions, we show that Definition~\ref{defn:zariski-closed} is compatible with the definition of Zariski-closed subsets from \cite[Def.\,5.7]{Scholze-diamond} in an appropriate sense. To make this precise, we  need to use the diamondification functor from \cite[Def.~15.5]{Scholze-diamond}. 

\begin{lemma}\label{lemma:different-ZC} Let $X=\Spa(A, A^+)$ be a strongly noetherian Tate affinoid over $\Spa(\Q_p, \Z_p)$, let $Z\hookrightarrow X$ be a Zariski-closed immersion (in the sense of Definition~\ref{defn:zariski-closed}), and let $Y=\Spa(R, R^+)$ be an affinoid perfectoid space with a morphism $Y \to X$. Then the fiber product $(Z')^\diam \coloneqq Z^\diam \times_{X^\diam} Y^\diam$ is represented a Zariski-closed perfectoid subspace of $Y^\diam$ (in the sense of \cite[Def.\,5.7]{Scholze-diamond}).
\end{lemma}

We note that the pre-adic space $Z'$ might not be a perfectoid space in general. For example, if we put $X \coloneqq \Spa(\mathbf{C}_p\langle T\rangle, \mathcal{O}_{\bf{C}_p}\langle T\rangle)$, $Z \coloneqq \Spa(\mathbf{C}_p, \mathcal{O}_{\bf{C}_p}) \hookrightarrow X$, the morphism $Z\hookrightarrow X$ the closed immersion of the origin into the closed unit disc, and $Y\coloneqq \Spa(\mathbf{C}_p\langle T^{ 1/p^\infty}\rangle, \mathcal{O}_{\bf{C}_p}\langle T^{ 1/p^\infty}\rangle)$, then $Z'\coloneqq Z\times_X Y$ is not an affinoid perfectoid space because $T^{1/p}\in \O(Z)=\bf{C}_p\langle T^{1/p^\infty}\rangle/(T)$ is a non-trivial nilpotent. Nevertheless, the diamond $(Z')^\diam \simeq \Spa(\mathbf{C}_p, \mathcal{O}_{\bf{C}_p})^\diam$ is representable by an affinoid perfectoid space. Furthermore, Lemma~\ref{lemma:different-ZC} guarantees that the diamond $(Z')^\diam$ is always representable by an affinoid perfectoid space.

\begin{proof}
    By \cite[Cor.\,B.6.9]{quotients}, there is a finitely generated ideal $I\subset A$ such that $Z=\Spa(A/I, (A^+/I\cap A^+)^c)$. Choose some generators $I=(f_1, \dots, f_m)$ and a pseudo-uniformizer $\varpi\in A^+$. Then we have
    \[
    Z \sim \lim_n X\left(|f_1| \leq |\varpi|^n, \dots, |f_m| \leq |\varpi|^n\right),
    \]
    where $\sim$ stands for the $\sim$-limit in the sense of \cite[Def.\,2.4.2]{H3} or \cite[Def.\,2.4.1]{SW}. Then the proof of \cite[Prop.~2.4.5]{SW} implies that $Z^\diam \xr{\sim} \lim_n \Big(X\left(|f_1| \leq |\varpi|^n, \dots, |f_m| \leq |\varpi|^n\right)^\diam\Big)$. Since diamondification commutes with fiber products (see \cite[Prop.~6.1.6(6)]{almost-coherent}), we conclude that 
    \begin{equation}\label{eqn:limit-diamonds}
    (Z')^\diam \simeq \lim_n \Big(Y\left(|f_1| \leq |\varpi|^n, \dots, |f_m| \leq |\varpi|^n\right)^\diam\Big).
    \end{equation}
    Now we note that each diamond
    \[
    Y(|f_1| \leq |\varpi|^n, \dots, |f_m| \leq |\varpi|^n)^\diam \simeq Y^\diam(|f_1| \leq |\varpi|^n, \dots, |f_m| \leq |\varpi|^n)
    \]
    is represented by an affinoid perfectoid space because it is a rational subdomain inside an affinoid perfectoid space $Y^\diam$. Therefore, \cite[Prop.~6.5]{Scholze-diamond} and Equation~(\ref{eqn:limit-diamonds}) imply that $(Z')^\diam$ is represented by an affinoid perfectoid space defined by the equations $|f|=0 \in Y$ for $f\in IR$. In particular, it is represented by a Zariski-closed immersion (see \cite[Def.~5.7(i)]{Scholze-diamond})
\end{proof}

Now we concentrate on a particular class of Zariski-closed immersions:

\begin{defn}\label{defn:lci-immersion} A Zariski-closed immersion $i\colon X \hookrightarrow Y$ of strongly noetherian Tate affinoids is a {\it regular immersion of pure codimension $c$} if the ideal of immersion $\cal{I}(Y) \subset \cal{O}_Y(Y)$ is generated by a regular sequence $(g_{i,1}, \dots, g_{i,c}) \subset \cal{O}_Y(Y)$.

A Zariski-closed immersion $i\colon X \hookrightarrow Y$ is an {\it lci immersion (of pure codimension $c$)} if there is an open affinoid covering $Y=\sqcup_{i\in I} U_i$ by strongly noetherian Tate affinoids such that the base change $X_{U_i} \hookrightarrow U_i$ is a regular immersion (of pure codimension $c$) for every $i\in I$.

A Zariski-closed immersion $i\colon X \hookrightarrow Y$ is an {\it effective Cartier divisor} if it an lci immersion of pure codimension $1$.
\end{defn}

\begin{rmk}\label{rmk:eff-cartier-usual} We note that a Zariski-closed immersion $i\colon X \hookrightarrow Y$ is an effective Cartier divisor if and only if the ideal sheaf $\cal{I}_X$ is an invertible $\O_Y$-module. This description is often more convenient in practice.
\end{rmk}

Now we begin verifying some basic properties of lci immersions:

\begin{lemma}\label{lemma:base-change-general} Let $Y=\Spa(A, A^+)$ be a strongly noetherian Tate affinoid, let $i\colon X\hookrightarrow Y$ be a regular immersion of pure codimension $c$, and let $U=\Spa(A_U, A_U^+)\subset Y$ be an open affinoid. Then the base change $i_T\colon X_U \hookrightarrow U$ is also a regular immersion of pure codimension $c$.
\end{lemma}
\begin{proof}
    We first note that \cite[Cor.\,B.6.9]{quotients} implies that $\O_X(X)=A/I$ for an ideal $I\subset A$. Then \cite[Lemma B.6.7]{quotients} guarantees that the ideal of $i_T$ is equal to the ideal $IA_U\subset A_U$. Finally, the fact that $IA_U\subset A_U$ is generated by a regular sequence of length $c$ follows from flatness of $A \to A_U$ (see \cite[Lemma B.4.3]{quotients}) and \cite[\href{https://stacks.math.columbia.edu/tag/00LM}{Tag 00LM}]{stacks-project}.  
\end{proof}

\begin{lemma}\label{lemma:flat-base-change} Let $i\colon X \hookrightarrow Y$ be an lci immersion (of pure codimension $c$), and let $f\colon Y'\to Y$ be a flat morphism of locally noetherian analytic adic spaces. Then the base change 
\[
i'\colon X'\coloneqq Y'\times_Y X \hookrightarrow Y'
\]
is also an lci immersion (of pure codimension $c$).
\end{lemma}
\begin{proof}
    Lemma~\ref{lemma:base-change-general} ensures that the question is local on $X$, $Y$, and $Y'$. So we can assume that $X=\Spa(B, B^+)$, $Y=\Spa(A, A^+)$, and $Y'=\Spa(C, C^+)$ are strongly noetherian Tate affinoids, and $X \to Y$ is a regular immersion of pure codimension $c$. Then Definition~\ref{defn:lci-immersion} and \cite[Cor.\,B.6.9]{quotients} imply that $B=A/I$ for an ideal $I$ generated by a regular sequence of length $c$. Now \cite[Lemma B.6.7]{quotients} implies that it suffices to show that $IC$ is also generated by a regular sequence of length $c$. This follows from \cite[Lemma B.4.3]{quotients} and \cite[\href{https://stacks.math.columbia.edu/tag/00LM}{Tag 00LM}]{stacks-project}.
\end{proof}

\begin{rmk}\label{rmk:smooth-flat} Any smooth morphism of locally noetherian analytic adic spaces is flat by \cite[Remark B.4.7]{quotients}. In particular, Lemma~\ref{lemma:flat-base-change} holds for any smooth morphism $f\colon Y' \to Y$.
\end{rmk}

\begin{lemma}\label{lemma:base-change-lci} Let $i\colon X \hookrightarrow Y$ be an lci immersion of pure codimension $c$, and let $f\colon Y'\to Y$ be a morphism of locally noetherian analytic adic spaces. Suppose that the base change 
\[
i'\colon X'\coloneqq Y'\times_Y X \hookrightarrow Y'
\]
is an lci immersion of pure codimension $c$. Then the natural morphism
\[
f^*\cal{I}_X \to \cal{I}_{X'}
\]
is an isomorphism, where $\cal{I}_X$ and $\cal{I}_{X'}$ are the ideal sheaves of the Zariski-closed immersions $i$ and $i'$ respectively.
\end{lemma}
\begin{proof}
    Arguing as in the proof of Lemma~\ref{lemma:flat-base-change}, we reduce the question to proving the following claim: \smallskip
    
    {\it Claim: Let $A$ be a noetherian ring, $I\subset A$ an ideal generated by a regular sequence of length $c$, and $A \to B$ is a ring homomorphism such that $IB$ is still generated by a regular sequence of length $c$. Then $I\otimes_A B \to IB$ is an isomorphism.} \smallskip
    
    By induction, one can assume that $c=1$. In this case, $I=(g)\subset A$ is a free $A$-module of rank-$1$. The assumption on $B$ tells us that $gB$ is a free $B$-module of rank-$1$. Therefore, $I\otimes_A B \to IB$ is a surjection of free $B$-modules of rank-$1$. Hence it is an isomorphism. 
\end{proof}

Our next goal is to give some interesting examples of lci immersions. We also will give more examples in the next section.

\begin{lemma}\label{lemma:local-structure-smooth-smooth} Let $S$ be a locally noetherian analytic adic space, let $f_X\colon X \to S$ be a smooth morphism of pure dimension $d_X$, let $f_Y\colon Y \to S$ be a smooth morphism of pure dimension $d_Y$, and let $i\colon X \hookrightarrow Y$ be a Zariski-closed immersion of adic $S$-spaces. Then, for each point $x\in X$, there is an open affinoid $x\in U_x\subset Y$ and an \'etale morphism $h\colon U_x \to \bf{D}^{d_Y}_S$ such that there is a Cartesian diagram
\[
    \begin{tikzcd}
        U_x\cap  X \arrow[r, hook, "i|_{U_x\cap X}"] \arrow{d}{h|_{U_x\cap X}} & U_x \arrow{d}{h} \\
        \bf{D}^{d'}_S \arrow[r, hook, "j"] & \bf{D}^{d}_S,
    \end{tikzcd}
\]
where $j\colon \bf{D}^{d_X}_S \to \bf{D}^{d_Y}_S$ is the inclusion of $\bf{D}^{d_X}_S$ into $\bf{D}^{d_Y}_S$ as the vanishing locus of the first $d_Y-d_X$ coordinates.
\end{lemma}
\begin{proof}
    Let us denote by $\cal{I}$ the ideal sheaf of the Zariski-closed immersion $i\colon X \hookrightarrow Y$. The claim is local on $S$, so we clearly can assume that $S$ is a Tate affinoid with a pseudo-uniformizer $\varpi$. \smallskip
    
    Now \cite[Prop.\,1.6.9(ii)]{H3} implies that there exist generators 
    \[
    g_{1, x}, \dots, g_{d', x}\in \cal{I}_x
    \]
    such that the elements $\{d_{Y/S}(g_{1, x}), \dots, d_{Y/S}(g_{d', x})\} \in \Omega^1_{Y/S, x}$ can be extended to a basis of an $\O_{X, x}$-module $\Omega^1_{Y/S, x}$. Using \cite[(1.6.2)]{H3}, we can assume that the extended basis is of the form 
    \[
    \{d_{Y/S}(g_{1, x}), \dots, d_{Y/S}(g_{d', x}), d_{Y/S}(g_{d'+1, x}), \dots, d_{Y/S}(g_{d, x})\} \in \Omega^1_{Y/S, x}
    \]
    for an integer $d\geq d'$ and some elements $g_{d'+1, x}, \dots, g_{d, x}\in \O_{Y, x}$. \smallskip
     
    Using that the $\O_Y$-modules $\cal{I}$ and $\Omega^1_{Y/S}$ are coherent (see Definition~\ref{defn:zariski-closed} and \cite[Prop.~1.6.9(i)]{H3} respectively), we can find an affinoid open neighborhood $x\in U_x = \Spa(B, B^+) \subset Y$ and functions $g_1, \dots, g_{d'}\in \cal{I}(U_x)\subset B$ and $g_{d'+1}, \dots, g_{d}\in \O_Y(U_x)=B$ such that we have equality of stalks $(g_i)_x = g_{i, x}$ for any $i=1, \dots, d$, the functions $g_1, \dots, g_{d'}$ generate the ideal $\cal{I}(U_x)$, and the set $\{d_{Y/S}(g_{1}), \dots, d_{Y/S}(g_{d'}), \dots, d_{Y/S}(g_d)\} \in \Omega^1_{Y/S}(U_x)$ forms a basis of the $\O_Y(U_x)$-module $\Omega^{1}_{Y/S}(U_x)$. In particular, these assumptions imply that $X\cap U_x\subset U_x$ is the vanishing locus of the functions $g_1, \dots, g_{d'}$.  \smallskip
    
    We can simultaneously multiply $g_1, \dots, g_d$ by some power of $\varpi$ to assume that $g_i\in B^+$ and then consider the unique $\O_S(S)$-linear morphism
    \[
    h^\sharp \colon \left(\O_S(S)\langle T_1, \dots, T_d\rangle, \O_S(S)^+\langle T_1, \dots, T_d\rangle\right) \to (B, B^+)
    \]
    sending $T_i$ to $g_i$. It defines a morphism of $S$-adic spaces
    \[
    h\colon U_x \to \bf{D}^d_S.
    \]
    This morphism is \'etale by virtue of \cite[Prop.\,1.6.9(iii)]{H3} and our choice of $g_1, \dots, g_d$. Now the construction of $h$ and \cite[Lemma B.6.7]{quotients} imply that $h$ fits into the Cartesian diagram
    \begin{equation}\label{eqn:locally-smooth-smooth}
    \begin{tikzcd}
        U_x\cap  X \arrow[r, hook, "i|_{U_x\cap X}"] \arrow{d}{h|_{U_x\cap X}} & U_x \arrow{d}{h} \\
        \bf{D}^{d'}_S \arrow[r, hook, "j"] & \bf{D}^{d}_S,
    \end{tikzcd}
    \end{equation}
    where $j$ is the inclusion of $\bf{D}^{d'}_S$ into $\bf{D}^{d}_S$ as the vanishing locus of the first $d-d'$ coordinates. We are only left to show that $d=d_Y$ and $d'=d_X$. This follows from Corollary~\ref{cor:dimension-of-disc}.
\end{proof}

\begin{rmk} In general, a similar argument shows that, for any smooth morphism $f\colon X \to S$ and a point $x\in X$, there is an open $x\in U$ and an integer $d$ such that $f|_U$ factors as a composition
\[
U \xr{g} \bf{D}^d_S \xr{\pi} S,
\]
where $g$ is an \'etale map and $\pi$ is the standard projection. In particular, analytically locally on the source, any smooth morphism is relatively pure of some dimension $d$. 
\end{rmk}

\begin{cor}\label{cor:smooth-smooth-lci} In the notation of Lemma~\ref{lemma:local-structure-smooth-smooth}, $i$ is an lci immersion of pure codimension $d_Y-d_X$. In particular, a section $s\colon S \to X$ of a separated smooth morphism $f\colon X \to S$ (of pure relative dimension $d$) is an lci immersion of (pure codimension $d$).
\end{cor}
\begin{proof}
    The first claim directly from Lemma~\ref{lemma:local-structure-smooth-smooth}, Lemma~\ref{lemma:flat-base-change}, and Remark~\ref{rmk:smooth-flat}. The ``in particular'' part follows from the previous claim if we can show that a section of a separated morphism is a Zariski-closed immersion. This, in turn, follows from the pullback diagram
    \[
    \begin{tikzcd}
    S \arrow{d}{s}\arrow{r}{s} & X\times_S S \arrow{d}{\rm{id}_X\times s}\\
    X\arrow[r, hook, "\Delta_{X/S}"] & X\times_S X,
    \end{tikzcd}
    \]
    the fact that $\Delta_{X/S}$ is a Zariski-closed immersion (see \cite[Cor.\,B.7.4]{quotients}), and the fact that Zariski-closed immersions are closed under pullbacks (see \cite[Cor.\,B.6.10]{quotients}).
\end{proof}

\section{Relative analytification}\label{section:analytification}

In this section, we consider the functor of relative analytification and show some of its basic properties. For the rest of this section, we fix a strongly noetherian Tate affinoid space $S=\Spa(A, A^+)$. \smallskip

We recall that the universal property of affine schemes (see \cite[\href{https://stacks.math.columbia.edu/tag/01I1}{Tag 01I1}]{stacks-project}) says that 
\[
\rm{Map}_{\bf{LRS}}(S, \Spec A)=\rm{Map}_{\bf{Rings}}(\O_S(S), A)=\rm{Map}_{\bf{Rings}}(A, A),
\]
where $\bf{LRS}$ is the category of locally ringed spaces. In particular, the identity morphism $\rm{id}_A$ defines a morphism of locally ringed spaces 
\[
c_S \colon S=\Spa(A, A^+) \to \Spec A.
\]

The main goal of this section is to study ``analytification'' along this map. More precisely, we give the following definition: 

\begin{defn}\label{defn:relative-analytification} A {\it relative analytification} of a locally finite type $A$-scheme $X$ is an adic $S$-space $X^{\an/S}\to S$ with a morphism of locally ringed $\Spec A$-spaces $c_{X/S}\colon X^{\an/S} \to X$ such that, for every adic $S$-space $U$, $c_{X/S}$ induces a bijection
\[
\rm{Map}_{\bf{Adic}_{/S}}(U, X^{\an/S}) \simeq \rm{Map}_{\bf{LRS}_{/\Spec A}}(U, X).
\]
\end{defn}

\begin{rmk} Clearly, a relative analatytification is unique if it exists. Furthermore, \cite[Prop.\,3.8]{H1} implies that it always exists for a locally finite type $A$-scheme $X$ and $X^{\an/S}$ is locally of finite type over $S$ in this case.
\end{rmk}

\begin{rmk}\label{rmk:analytification-proper}(\cite[Lemma 5.7.3]{H3}) If $X$ is a proper $\O_S(S)$-scheme, then $X^{\an/S}$ is a proper adic $S$-space. 
\end{rmk}

Our first goal is to show that the analytification morphism $c_{X/S}$ is flat. Before doing this, we need to examine some examples in more detail. 

\begin{exmpl}\label{example:analytification-of-affine-space} Let $\varpi\in A^+$ be a pseudo-uniformizer, let $d$ be a positive integer, let $\bf{A}^d_A \coloneqq \Spec A[T_1, \dots, T_d]$ be the relative affine space over $A$, and let the adic space 
\[
\mathbf{D}^d_S(|\frac{1}{\varpi^n}|) \coloneqq \Spa(A\langle \varpi^n\cdot T_1, \dots, \varpi^n\cdot T_d \rangle, A^+\langle \varpi^n \cdot T_1, \dots, \varpi^n\cdot T_d\rangle)
\]
be the relative closed disk of ``radius $|\frac{1}{\varpi^n}|$''. Then there is an isomorphism of adic $S$-spaces
\[
\bf{A}_A^{d, \an/S} \simeq \bigcup_{n\geq 0} \bf{D}^d_S\big(|\frac{1}{\varpi^n}|\big). 
\]
\end{exmpl}
\begin{proof}
    Since any adic space space is locally affinoid, it suffices to show that that there is a functorial bijection
    \[
    \rm{Map}_{\bf{Adic}_{/S}}\Big(U, \bigcup_{n\geq 0}\bf{D}^d_S\big(|\frac{1}{\varpi^n}|\big) \Big) \simeq \rm{Map}_{\bf{LRS}_{/\Spec A}}\big(U, \bf{A}^d_A\big)
    \]
    for any {\it affinoid} adic $S$-space $U=\Spa(B, B^+)$. Now \cite[\href{https://stacks.math.columbia.edu/tag/01I1}{Tag 01I1}]{stacks-project} implies that 
    \[
    \rm{Map}_{\bf{LRS}_{/\Spec A}}\big(U, \bf{A}^d_A\big) \simeq \O_U(U)^d \simeq B^d.
    \]
    Therefore, it suffices to construct a functorial isomorphism $\rm{Map}_{\bf{Adic}_{/S}}\Big(U, \bigcup_{n\geq 0}\bf{D}^d_S\big(|\frac{1}{\varpi^n}|\big) \Big) \simeq B^d$. Using that each $U$ is quasi-compact, each $\bf{D}^d_S(|\frac{1}{\varpi^n}|)$ is (abstractly) isomorphic to the closed unit disk over $S$, and the universal property of the closed unit disc (see \cite[Lem.~3.5]{H1}), we conclude that 
    \begin{align*}
    \rm{Map}_{\bf{Adic}_{/S}}\Big(U, \bigcup_{n\geq 0}\bf{D}^d_S\big(|\frac{1}{\varpi^n}|\big) \Big) & \simeq \colim_{n\geq 0} \rm{Map}_{\bf{Adic}_{/S}}\Big(U, \bf{D}^d_S\big(|\frac{1}{\varpi^n}|\big) \Big) \\
    & \simeq \colim \big((B^+)^d \xr{\cdot \varpi} (B^+)^d \xr{\cdot \varpi} \dots\big) \\
    & \simeq (B^+)^d\big[\frac{1}{\varpi}\big] \simeq B^d,
    \end{align*}
    where the last isomorphism uses that $B^+[\frac{1}{\varpi}] \simeq B$ since $\varpi$ is a pseudo-uniformizer in $B$. This finishes the proof. 
\end{proof}

\begin{exmpl}\label{example:affinoid-analytification} Let $\varpi\in A^+$ be a pseudo-uniformizer, let $d$ be a positive integer, let $X \coloneqq \rm{V}(f_1, \dots, f_s)\subset \bf{A}^d_A$ be the vanishing locus of polynomials $f_1, \dots, f_s\in A[T_1, \dots, T_d]$, and let $\rm{V}^{\an}(f_1, \dots, f_s)\subset \bf{A}^{d, \an/S}_S$ be the vanishing locus of the same functions in the relative analytic affine space. Then there is an isomorphism of adic $S$-spaces
\[
X^{\an/S} \simeq \rm{V}^{\an}(f_1, \dots, f_s)\subset \bf{A}^{d, \an/S}_S. 
\]
\end{exmpl}
\begin{proof}
    Using Example~\ref{example:analytification-of-affine-space} and \cite[\href{https://stacks.math.columbia.edu/tag/01HP}{Tag 01HP}]{stacks-project}, we conclude that, for any affinoid adic $S$-space $U=\Spa(B, B^+)$, the sets $\rm{Map}_{\bf{Adic}_{/S}}\big(U, \rm{V}^{\an}(f_1, \dots, f_s)\big)$ and $\rm{Map}_{\bf{LRS}_{/\Spec A}}\big(U, X\big)$ can be functorailly identified with $d$-tuples $h_1, \dots, h_d\in B$ such that 
    \[
    f_1(h_1, \dots, h_d)=f_2(h_1, \dots, h_d) = \dots = f_s(h_1, \dots, h_d)=0,
    \]
    where the equality takes place in the ring $B$. In particular, this implies that $X^{\an/S} \simeq \rm{V}^{\an}(f_1, \dots, f_s)$. 
\end{proof}

\begin{rmk} We note that Example~\ref{example:affinoid-analytification} and a standard gluing argument could be used to show the existence of relative analytification for any locally finite type $A$-scheme $X$. This reproves \cite[Prop.\,3.8]{H1}. 
\end{rmk}

\begin{lemma}\label{lemma:flatness} Let $X$ be a locally finite type $A$-scheme, then the analytification morphism $c_{X/S} \colon (X^{\an/S}, \O_{X^{\an/S}}) \to (X, \O_X)$ is a flat morphism of locally ringed spaces, i.e., for any point $x\in X^{\an/S}$, the natural morphism $c_{X/S, x}^\sharp \colon \O_{X, c_{X/S}(x)} \to \O_{X^{\an/S}, x}$ is flat. 
\end{lemma}
\begin{proof}
    The question is clearly local on $X$, so we can assume that $X$ is affine. We choose a closed immersion $X=\mathrm{V}(f_1,\dots, f_s) \hookrightarrow \mathbf{A}^d_{A}$. Pick a point $x\in X^{\an/S}$. Then Example~\ref{example:affinoid-analytification} and \cite[Cor.~B.6.8]{quotients} imply that $\O_{X^{\an/S}, x}\simeq \O_{\bf{A}_A^{d, \an/S},x}/(f_1, \dots, f_s)$ and, likewise, we have $\O_{X, c_{X/S}(x)} \simeq \O_{\bf{A}^d_A, c_{X/S}(x)}/(f_1, \dots, f_s)$. Therefore, it suffices to prove the result for $X=\bf{A}^d_A$. In this case, Example~\ref{example:analytification-of-affine-space} implies that it suffices to show that the natural morphism
    \[
    c_n\coloneqq c_{\bf{A}^d_A/S}|_{\bf{D}^d_S\big(|\frac{1}{\varpi^n}|\big)} \colon \bf{D}^d_S\big(|\frac{1}{\varpi^n}|\big) \to \bf{A}^d_A
    \]
    is flat for any integer $n\geq 0$. Using the commutative diagram
    \[
    \begin{tikzcd}
    \bf{D}^d_S\big(|\frac{1}{\varpi^n}|\big) \arrow{r}{c_n} \arrow{d}{\cdot \varpi^n} \arrow[d, swap, "\wr"] & \bf{A}^d_A \arrow{d}{\cdot \varpi^n} \arrow[d, swap, "\wr"]\\
    \bf{D}^d_S \arrow{r}{c_0} & \bf{A}^d_A,
    \end{tikzcd}
    \]
    we reduce the question to showing that $c_0\colon \bf{D}^d_S \to \bf{A}^d_A$ is flat. This map can be written as the following composition
    \[
    \Spa(A\langle T_1, \dots, T_d\rangle, A^+\langle T_1, \dots, T_d\rangle) \xr{\alpha} \Spec A\langle T_1, \dots, T_d\rangle \xr{\beta} \Spec A[T_1, \dots, T_d].
    \]
    Now \cite[Prop.~II.6.6.3]{FujKato} implies that $\alpha$ is flat (alternatively, one can deduce it from \cite[(II.1) (iv) on page 530]{H1}). To see that $\beta$ is flat, it suffices to show that the map $A^+[T_1, \dots, T_d] \to A^+\langle T_1, \dots, T_d\rangle$ is flat. This follows from \cite[Prop.~0.8.2.18]{FujKato} and \cite[Thm.~0.8.2.19]{FujKato}.     
\end{proof}

\begin{cor}\label{cor:lci-relative-analytification} Let $X$ and $Y$ be locally finite type $A$-schemes, and let $i\colon X \to Y$ be an lci closed immersion of pure codimension $c$. Then its relative analytification $i^{\an/S} \colon X^{\an/S} \to Y^{\an/S}$ is an lci closed immersion of pure codimension $c$.
\end{cor}

We note that \cite[Proposition 5.5]{Guo-Li} proves a stronger claim under the assumption that $A$ is a $K$-affinoid algebra over a non-archimedean field $K$.
\begin{proof}
    The question is local on $Y$, so we can assume that $Y=\Spec R$ is affine and $X=\Spec R/I$ for an ideal $I=(f_1, \dots, f_c)$ generated by a regular sequence of length $c$. Then, arguing inductively, we can assume that $X=\Spec R/(f)$ for a non-zero regular element $f\in R$. In this case, we wish to show that $i^{\an/S} \colon X^{\an/S} \to Y^{\an/S}$ is an effecitve Cartier divisor. For this, we note that Remark~\ref{rmk:eff-cartier-usual} ensures that it suffices to show that $\cal{I}_{X^{\an/S}}$ is invertible, where $\cal{I}_{X^{\an/S}}$ is the ideal sheaf of the closed immersion $X^{\an/S}\subset Y^{\an/S}$. Now Lemma~\ref{lemma:flatness} implies that $c_{X/S}^*(\cal{I}_{X}) \simeq \cal{I}_{X^{\an/S}}$, where $\cal{I}_X$ is the ideal sheaf of $X\subset Y$. Since $X\subset Y$ is an effective Cartier divisor, we conclude that $\cal{I}_{X^{\an/S}} \simeq c_{X/S}^*(\cal{I}_{X})$ is an invertible $\O_{Y^{\an/S}}$-module. In other words, $X^{\an/S} \subset Y^{\an/S}$ is an effective Cartier divisor. This finishes the proof. 
\end{proof}

We also mention the following GAGA principles:

\begin{lemma}[Relative GAGA]\label{lemma:relative-GAGA} Let $X$ be a proper $A$-scheme and let $c_{X/S} \colon (X^{\an/S}, \O_{X^{\an/S}}) \to (X, \O_X)$ be the relative analytification morphism. Then the induced functors 
\[
c_{X/S}^* \colon D^b_{\coh}(X) \to D^b_{\coh}(X^{\an/S}),
\]
\[
c_{X/S}^* \colon \Coh(X) \to \Coh(X^{\an/S}),
\]
\[
c_{X/S}\colon \rPic(X) \to \rPic(X^{\an/S})
\]
are equivalences.
\end{lemma}
\begin{proof}
    The first part follows immediately from \cite[Thm.~II.9.5.1]{FujKato}. To see the second claim, we note that Lemma~\ref{lemma:flatness} implies that $c_{X/S}$ is flat. In other words, $c_{X/S}^*$ is $t$-exact. This implies that the equivalence $c_{X/S}^* \colon D^b_{\coh}(X) \xr{\sim} D^b_{\coh}(X^{\an/S})$ induces an equivalence of hearts $c_{X/S}^* \colon \Coh(X) \xr{\sim} \Coh(X^{\an/S})$. Finally, we note that $\rPic(X)$ (resp. $\rPic(X^{\an/S})$) can be identified with the category of invertible objects in the monoidal category $\Coh(X)^{\otimes}$ (resp. $\Coh(X^{\an/S})^{\otimes}$). Since $c_{X/S}^* \colon \Coh(X) \xr{\sim} \Coh(X^{\an/S})$ is a monoidal equivalence, we conclude that it induces an equivalence $c_{X/S}\colon \rPic(X) \xr{\sim} \rPic(X^{\an/S})$. 
\end{proof}

\section{Analytic Proj construction}\label{section:relative-proj}

This section is devoted to the discussion of the relative Proj construction in the world of adic spaces. For rigid-analytic spaces, this notion has been studied in \cite{conrad-ample}. \smallskip

For the next definition, we fix a locally noetherian analytic adic space $S$.

\begin{defn}\label{defn:locally-coherent} A {\it locally coherent graded $\O_S$-algebra} $\cal{A}_\bullet$ is a graded $\O_S$-algebra $\cal{A}_\bullet=\bigoplus_{d\geq 0} \cal{A}_d$ such that each $\cal{A}_d$ is a coherent $\O_S$-module, and $\cal{A}_\bullet$ is locally finitely generated as an $\O_S$-algebra.

Let $S$ be an affinoid. A {\it coherent graded $\O_S(S)$-algebra} $A_\bullet$ is a graded $\O_S(S)$-algebra $A_\bullet=\bigoplus_{d\geq 0} A_d$ such that each $A_d$ is a coherent $\O_S(S)$-module, and $A_\bullet$ is finitely generated as an $\O_S(S)$-algebra.
\end{defn}

Now we wish to show that there is an equivalence between locally coherent graded $\O_S$-algebras and coherent graded $\O_S(S)$-algebras for a strongly noetherian affinoid space $S$. For this, we will need the following lemma:

\begin{lemma}\label{lemma:flat-surj-faithfully-flat} Let $f\colon S'=\Spa(B, B^+) \to S=\Spa(A, A^+)$ be a flat (resp. surjective flat) morphism of strongly noetherian affinoid spaces. Then $f^\sharp\colon A \to B$ is flat (resp. faithfully flat).
\end{lemma}
\begin{proof}
    Flatness of $A \to B$ follows from \cite[Lemma B.4.3]{quotients}\footnote{\cite[Lemma B.4.2 and B.4.3]{quotients} are formulated for Tate affinoids. However, the same proofs work for analytic affinoids. One only needs to use \cite[Thm.\,1.4.14]{KedAr} in place of \cite[(II.1), (iv) on page 530]{H1} in the proof of \cite[Lemma B.4.2]{quotients}.}. Now we assume that $f$ is also surjective, and show that $f^\sharp$ is faithfully flat. It suffices to show that $\Spec B \to \Spec A$ is surjective onto the closed points of $\Spec A$. This follows from the fact that, for any maximal ideal of $\mathfrak{m}\subset A$, there is a point $v\in \Spa(A, A^+)$ such that $\rm{supp}(v)=\mathfrak{m}$ (see \cite[Lemma 1.4]{H1}) and surjectivity of $\Spa(B, B^+) \to \Spa(A, A^+)$. 
\end{proof}

\begin{lemma} Let $S$ be a strongly noetherian affinoid. Then $\Gamma(S, -)$ defines an equivalence
\begin{equation*}
\Gamma(S, -)\colon \left\{\begin{array}{c}\text{locally coherent graded} \\  \O_S\text{-algebras}\end{array}\right\} \xr{\sim} \left\{\begin{array}{l} \text {coherent graded} \\ \O_S(S)\text{-algebras}\end{array}\right\}.
\end{equation*}
\end{lemma}
\begin{proof}
    The proof essentially follows from Lemma~\ref{thm:huber-coherent-sheaves}. One easily sees that $\widetilde{(-)}$ provides a quasi-inverse to $\Gamma(S, -)$ provided that, for a locally coherent graded $\O_S$-algebra $\cal{A}_\bullet$, the $\O_S$-algebra
    \[
    \Gamma(S, \cal{A}_\bullet)
    \]
    is naturally graded and coherent as a graded $\O_S(S)$-algebra. For the purposes of proving the first claim, it suffices to show $\Gamma(S, -)$ commutes with infinite direct sums. This follows from spectrality of $S$ and \cite[\href{https://stacks.math.columbia.edu/tag/009F}{Tag 009F}]{stacks-project}. \smallskip
    
    Now we need to show that $\Gamma(S, \cal{A}_\bullet)$ is a coherent graded $\O_S(S)$-algebra for any locally coherent graded $\O_S$-algebra. The locally coherent assumption together with Lemma~\ref{thm:huber-coherent-sheaves}, Lemma~\ref{lemma:pullback-coherent}(\ref{lemma:pullback-coherent-2}), and Lemma~\ref{lemma:flat-surj-faithfully-flat} imply that there is a morphism of strongly noetherian affinoid adic spaces\footnote{In fact, $S'$ is a finite disjoint union of rational subdomains in $S$.} $S' \to S$ such that the ring homomorphism $\O_S(S) \to \O_{S'}(S')$ is faithfully flat and 
    \[
    \Gamma(S, \cal{A}_\bullet) \otimes_{\O_S(S)} \O_{S'}(S') 
    \]
    is a finitely generated $\O_{S'}(S')$-algebra. Therefore, \cite[\href{https://stacks.math.columbia.edu/tag/00QP}{Tag 00QP}]{stacks-project} ensures that $\Gamma(S, \cal{A}_\bullet)$ is a finite generated $\O_S(S)$-algebra. 
\end{proof}

For the next definition, we fix a stronly noetherian affinoid $S$, a locally coherent graded $\O_S$-algebra $\cal{A}$, and a corresponding coherent graded $\O_S(S)$-algebra $A_\bullet$.

\begin{defn}\label{defn:proj-affinoid} The {\it analytic relative Proj space} 
\[
\ud{\rm{Proj}}^{\an}_S \cal{A}_\bullet \coloneqq \left(\ud{\rm{Proj}}_{\Spec \O_S(S)} A_\bullet\right)^{\an/S}
\]
is the relative analytification of the algebraic (relative) Proj scheme\footnote{See \cite[\textsection 2]{EGA2} for a detailed discussed of the algebraic Proj construction. In particular, use \cite[Prop.\,(2.7.1)]{EGA2} to ensure that $\ud{\rm{Proj}}_{\Spec \O_S(S)} A_\bullet$ is a finite type $\O_S(S)$-scheme, so its relative analytification is well-defined.}. 
\end{defn}

\begin{lemma}\label{lemma:proj-base-change} Let $f\colon S' \to S$ be a  morphism of strongly noetherian affinoids, and let $\cal{A}_\bullet$ be a locally coherent graded $\O_S$-algebra. Then there is a natural isomorphism
\[
\psi_{S, S'}\colon \ud{\rm{Proj}}_{S'}^{\an} \left(f^*\cal{A}_\bullet\right) \xr{\sim}\ud{\rm{Proj}}_{S}^{\an} \left(\cal{A}_\bullet\right) \times_S S'.
\]
Furthermore, if $g\colon S'' \to S'$ is another morphism of strongly noetherian affinoids, then the diagram
\[
\begin{tikzcd}
    \ud{\rm{Proj}}_{S''}^{\an} \left(g^*\left(f^*\cal{A}_\bullet\right)\right) \arrow{d}{\wr}  \arrow{r}{\psi_{S', S''}} & \ud{\rm{Proj}}_{S'}^{\an} \left(f^*\cal{A}_\bullet\right) \times_{S'} S'' \arrow{r}{\psi_{S, S'}\times \rm{id}} & \ud{\rm{Proj}}_{S}^{\an} \left(\cal{A}_\bullet\right) \times_S S'\times_{S'} S'' \arrow{d}{\wr} \\
    \ud{\rm{Proj}}_{S''}^{\an} \left(\left(f\circ g\right)^*\cal{A}_\bullet\right) & & \ud{\rm{Proj}}_{S}^{\an} \left(\cal{A}_\bullet\right) \times_S S'' \arrow{ll}{\psi_{S'', S}}
\end{tikzcd}
\]
commutes.
\end{lemma}
\begin{proof}
    Let $A_\bullet$ be a coherent graded $\O_S(S)$-algebra corresponding to $\cal{A}_\bullet$. Then, after unravelling the definitions, it suffices to show that there is a natural isomorphism of $\O_{S'}(S')$-schemes
    \[
    \ud{\rm{Proj}}_{S'} \left(A_\bullet \otimes_{\O_S(S)} \O_{S'}(S')\right) \simeq \ud{\rm{Proj}}_S \left(A_\bullet\right) \times_S S' 
    \]
    that satisfies the ``cocycle'' formula. This follows from  \cite[\href{https://stacks.math.columbia.edu/tag/01N2}{Tag 01N2}]{stacks-project} and \cite[\href{https://stacks.math.columbia.edu/tag/01MZ}{Tag 01MZ}]{stacks-project}. 
\end{proof}

\begin{lemma}\label{lemma:gluing-proj} Let $S$ be a locally noetherian analytic adic space, and let $\cal{A}_\bullet$ be a locally coherent graded $\O_S$-algebra. Then there is an (essentially unique) analytic adic $S$-space 
\[
\pi\colon \ud{\rm{Proj}}^{\an}_S \cal{A}_\bullet \to S
\]
with the following properties: 
\begin{enumerate}
    \item for every affinoid $U\subset S$ there exists an isomorphism $i_U\colon \pi^{-1}(U) \xr{\sim} \ud{\rm{Proj}}^{\an}_U \cal{A}_{\bullet}|_U$, and 
    \item for affinoid opens $V\subset U \subset S$ the composition
    \[
    \ud{\rm{Proj}}^{\an}_V \cal{A}_{\bullet}|_V \xr{i_V^{-1}} \pi^{-1}(V) \xr{\sim} \pi^{-1}(U)\times_{U} V \xr{i_{U}\times_{U} V} \ud{\rm{Proj}}^{\an}_{U} \cal{A}_{\bullet}|_{U} \times_{U} V
    \]
    is equal to $\psi_{U, V}$ from Lemma~\ref{lemma:proj-base-change}. 
\end{enumerate}
\end{lemma}
\begin{proof}
    This follows formally from Lemma~\ref{lemma:proj-base-change} and standard gluing arguments. 
\end{proof}

For the next definition, we fix a locally noetherian analytic adic space $S$ and a locally coherent graded $\O_S$-algebra $\cal{A}_\bullet$.

\begin{defn}\label{defn:proj-general} The {\it analytic relative Proj} of $\cal{A}_\bullet$ is the morphism 
\[
\ud{\rm{Proj}}^{\an}_S \cal{A}_\bullet \to S
\]
is the adic $S$-space constructed in Lemma~\ref{lemma:gluing-proj}.
\end{defn}

\begin{rmk}\label{rmk:proj-base-change} Lemma~\ref{lemma:proj-base-change} easily implies that the formation of analytic Proj commutes with arbitrary base change. More precisely, for any morphism $f\colon S' \to S$ of locally noetherian analytic adic spaces and a locally coherent graded $\O_S$-algebra $\cal{A}_\bullet$, there is a natural isomorphism
\[
\rm{Proj}^{\an}_{S'} \left(f^*\cal{A}_\bullet\right) \simeq \left(\rm{Proj}_S^{\an} \cal{A}_\bullet\right) \times_S S'. 
\]
\end{rmk}

\begin{rmk}\label{rmk:proj-proper} We note that $\ud{\rm{Proj}}^{\an}_S \cal{A}_\bullet$ is proper over $S$ by Remark~\ref{rmk:analytification-proper} and a combination of \cite[Prop.\,(3.1.9)(i), Prop.\,(3.1.10), and Thm.\,(5.5.3)(i)]{EGA2}.
\end{rmk}

\begin{rmk}\label{rmk:universal-line-bundle} Similarly to the algebraic situation, the analytic Proj-construction 
\[
P\coloneqq \ud{\rm{Proj}}_S^{\an} (\cal{A}_\bullet) \to S
\]
comes equipped with the coherent sheaf $\O_{P/S}(1)$. If $S$ is affinoid, one defines $\O_{P/S}(1)$ to be the relative analytification of the algebraic $\O(1)$. More precisely, we put $S^{\alg}\coloneqq \Spec \O_S(S)$ and $P^{\alg} = \ud{\rm{Proj}}_{\Spec \O_S(S)} A_\bullet \to \Spec \O_S(S)$ to be the algebraic Proj construction and define 
\[
\O_{P/S}(1)\coloneqq c_{P^{\alg}/S}^* \O_{P^{\alg}/S^{\alg}}(1),
\]
see Definition~\ref{defn:relative-analytification}. In general, one glues these line bundles locally on $S$. The formation of $\O_{P/S}(1)$ commutes with an arbitrary base change $S' \to S$. If $\cal{A}_\bullet$ is generated by $\cal{A}_1$, then $\O_{P/S}(1)$ is a line bundle. 
\end{rmk}

Now we give two particularly interesting examples of the analytic Proj construction:

\begin{defn}\label{defn:projective-bundle}
Let $\cal{E}$ be a vector bundle on $S$. The {\it projective bundle associated to $\cal{E}$} is the morphism
\[
\bf{P}_S(\cal{E})\coloneqq \ud{\rm{Proj}}_S^{\an}\left(\rm{Sym}_S^\bullet \cal{E}\right) \to S.
\]
\end{defn}

\begin{defn}\label{defn:blow-ups} Let $\cal{I}$ be a coherent ideal sheaf on locally noetherian analytic $X$ and $Z\subset X$ be the associated closed adic subspace. The {\it blow-up of $X$ along $Z$}, or {\it the blow-up of $X$ in the ideal sheaf $\cal{I}$}, is the morphism
\[
\rm{Bl}_Z(X)\coloneqq \ud{\rm{Proj}}^{\an}_X \big(\bigoplus_{d\geq 0} \cal{I}^d\big) \to X.
\]
\end{defn}

\begin{defn} For a blow-up $\pi \colon \rm{Bl}_Z(X) \to X$ of $X$ along $i\colon Z \to X$, the {\it exceptional divisor} $E \subset \rm{Bl}_Z(X)$ is defined to be $\rm{\pi}^{-1}(Z) = \ud{\rm{Proj}}^{\an}_Z  \big( \bigoplus_{d\geq 0} \cal{I}^d/\cal{I}^{d+1}\big) \subset \rm{Bl}_Z(X)$. 
\end{defn}

Now let $S=(A, A^+)$ be a strongly noetherian Tate affinoid affinoid, we wish to show that analytic blow-ups are compatible with the usual algebraic blow-ups. We start with the most general Proj-construction:

\begin{lemma}\label{lemma:proj-commutes-with-analytification} Let $S=\Spa(A, A^+)$ be as above, let $X$ be a locally finite type $A$-scheme with the relative analytification $c_{X/S} \colon (X^{\an/S}, \O_{X^{\an/S}}) \to (X, \O_X)$, and let $\cal{A}_\bullet$ be a locally coherent graded $\O_X$-algebra. Then there is a canonical isomorphism
\[
\psi_X \colon \ud{\rm{Proj}}_{X^{\an/S}}^{\an} \big(c_{X/S}^*(\cal{A}_\bullet)\big) \xrightarrow{\sim} \big(\ud{\rm{Proj}}_X \cal{A}_\bullet\big)^{\an/S}.
\]
such that, for every open $A$-subscheme $X'\subset X$, $\psi_X|_{X'^{\an/S}} =\psi_{X'^{\an/S}}$.
\end{lemma}
\begin{proof}
    First, we note that since we require $\psi_{X}$ to commute with open restrictions, we can construct them locally on $X$. In particular, we can assume that $X=\Spec B$ is an affine scheme. \smallskip
    
    Then we note that it suffices, for each affinoid open $U=\Spa(D, D^+) \subset X^{\an/S}$, to construct an isomorphim 
    \[
    \psi_U \colon \ud{\rm{Proj}}_{U}^{\an} \big(c_{X/S}^*(\cal{A}_\bullet)|_U\big) \xrightarrow{\sim} \big(\ud{\rm{Proj}}_X \cal{A}_\bullet\big)^{\an/S} \times_{X^{\an/S}} U
    \]
    such that $\psi_U|_V = \psi_V$ for each immersion of open affinoids $V\subset U\subset X^{\an/S}$. For this, we note that the composition $(U, \O_U) \to (X^{\an/S}, \O_{X^{\an/S}}) \to (X, \O_X)$ uniquely factors as
    \[
    \begin{tikzcd}
    (U, \O_U) \arrow{d}{j} \arrow{r}{c_{U}} & (\Spec D, \O_{\Spec D}) \arrow{d}{} \\
    (X^{\an/S}, \O_{X^{\an/S}}) \arrow{r}{c_{X/S}} & (X, \O_X),
    \end{tikzcd}
    \]
    where $c_U$ is the morphism described just above Definition~\ref{defn:relative-analytification}, and $\Spec D \to X=\Spec B$ is the morphism induced by the composition $B \to \O_{X^{\an/S}}(X^{\an/S}) \to \O_U(U)=D$. The universal properties of fiber products and relative analytifications imply that we have a functorial isomorphism 
    \begin{equation}\label{eqn:proj-analytification-base-change}
    \big(\ud{\rm{Proj}}_{X} \cal{A}_{\bullet} \times_X \Spec D \big)^{\an/U} \simeq \big(\ud{\rm{Proj}}_X \cal{A}_\bullet\big)^{\an/S} \times_{X^{\an/S}} U.
    \end{equation}
    Now we denote by $\cal{A}'_{\bullet}$ the pullback of $\cal{A}_{\bullet}$ to $\Spec D$. Using that the algebraic Proj construction commutes with base change, we conclude that isomorphism~(\ref{eqn:proj-analytification-base-change}) gives an isomorphism
    \begin{equation}\label{eqn:proj-analytification-base-change-2}
    \big(\ud{\rm{Proj}}_{\Spec D} \cal{A}'_\bullet\big)^{\an/U}\simeq \big(\ud{\rm{Proj}}_X \cal{A}_\bullet\big)^{\an/S} \times_{X^{\an/S}} U. 
    \end{equation}
    Using that $c_{X/S}^*(\cal{A}_\bullet)|_U \simeq c_U^*\cal{A}'_\bullet$ and Definition~\ref{defn:proj-affinoid}, we conclude that isomorphism~(\ref{eqn:proj-analytification-base-change-2}) becomes the desired functorial isomorphism 
    \[
        \psi_U \colon \ud{\rm{Proj}}_{U}^{\an} \big(c_{X/S}^*(\cal{A}_\bullet)|_U\big) \xrightarrow{\sim} \big(\ud{\rm{Proj}}_X \cal{A}_\bullet\big)^{\an/S} \times_{X^{\an/S}} U. \qedhere
    \]
\end{proof}

\begin{lemma}\label{lemma:blow-up-analytification} Let $S=\Spa(A, A^+)$ be as above, and let $Z\subset X$ be a closed immersion of locally finite type $A$-schemes with the relative analytification $Z^{\an/S} \to X^{\an/S}$. Then there is a canonical isomorphism of pairs
\[
\rm{Bl}_{Z^{\an/S}}(X^{\an/S}) \to \rm{Bl}_Z(X)^{\an/S}.
\]
Furthermore, this isomorphism identifies the exceptional divisor in $\rm{Bl}_{Z^{\an/S}}(X^{\an/S})$ with the analytification of the exceptional divisor in $\rm{Bl}_Z(X)$.
\end{lemma}
\begin{proof}
    Let us denote by $c_{X/S} \colon \big(X^{\an/S}, \O_{X^{\an/S}}\big) \to \big(X, \O_X\big)$ the relative analytification morphism, and by $\cal{I}_Z$ and $\cal{I}_{Z^{\an/S}}$ the ideal sheaves of the Zariski-closed immersions $Z\subset X$ and $Z^{\an/S} \subset X^{\an/S}$ respectively. Then Lemma~\ref{lemma:flatness} implies that $c_{X/S}^*(\cal{I}_Z^d) = \cal{I}_{Z^{\an/S}}^d$ for every $d\geq 1$. Now Lemma~\ref{lemma:proj-commutes-with-analytification} automatically implies the desired claim. 
\end{proof}

\section{Line bundles on the relative projective bundle}\label{section:line-bundles}

In this section, we study line bundles on a relative projective bundle $\bf{P}_S(\cal{E}) \to S$ for any locally noetherian analytic adic space $S$ and a vector bundle $\cal{E}$ on $S$. The main goal is to prove the following theorem: \smallskip

\begin{thm}\label{thm:relative-picard-projective-line} Let $S$ be a connected locally noetherian analytic adic space, let $\cal{E}$ be a vector bundle on $S$ of rank at least $2$, and let $f\colon P\coloneqq \bf{P}_S(\cal{E}) \to S$ be the corresponding projective bundle. Then the natural morphism
\[
\rm{Pic}\left(S\right) \bigoplus \Z \to \rm{Pic}\left(P\right)
\]
defined by the rule
\[
(\cal{L}, n) \mapsto f^*\cal{L} \otimes \O_{P/S}(n)
\]
is an isomorphism. 
\end{thm}

Let us begin with the case of a strongly noetherian Tate affinoid $S=\Spa(A, A^+)$ and a trivial vector bundle $\cal{E}=\O_S^{\oplus d+1}$ for a positive integer $d$. In this case, the relative projective space $\bf{P}^d_S \to S$ is the relative analytification of the relative algebraic projective space $\bf{P}^{d, \rm{alg}}_A \to \Spec A$. In particular, there is the analytification morphism
\[
i\colon \bf{P}^{d}_{S} \to \bf{P}^{d, \rm{alg}}_A.
\]

\begin{lemma}\label{lemma:GAGA} In the notation as above, the natural morphism
\[
\rm{Pic}\big(\bf{P}^{d, \rm{alg}}_A\big) \to \rm{Pic}\left(\bf{P}^d_S\right)
\]
is an isomorphism.
\end{lemma}
\begin{proof}
    This follows directly from Lemma~\ref{lemma:relative-GAGA}. 
\end{proof}

Lemma~\ref{lemma:GAGA} essentially proves Theorem~\ref{thm:relative-picard-projective-line} for the relative projective space over an affinoid base. However, in order to globalize the result, we will need to do some extra work. 

\begin{cor}\label{cor:case-of-a-field} Let $K$ be a non-archimedean field with an open bounded valuation subring $K^+\subset K$, let $S=\Spa(K, K^+)$, and let $f\colon  \bf{P}^d_S \to S$ be the relative projective space of relative dimension $d\geq 1$. Then the natural morphism
\[
\Z \to \rm{Pic}\left(\bf{P}^d_S\right),
\]
defined by the rule
\[
n \mapsto \O_{\bf{P}^d_S/S}(n),
\]
is an isomorphism. 
\end{cor}
\begin{proof}
    This follows directly from Lemma~\ref{lemma:GAGA} and the standard algebraic computation \[
    \rm{Pic}\left(\bf{P}^d_K\right)\simeq \Z[\O(1)]. \qedhere
    \]
\end{proof}

\begin{notation} Suppose that $f\colon \bf{P}_S(\cal{E}) \to S$ is a relative projective bundle over $S$ and $x\in S$ is a point. We will denote by $\bf{P}_x(\cal{E})$ the fiber product
\[
\bf{P}_S(\cal{E}) \times_S \Spa\left(K(x), K(x)^+\right)
\]
and call it the {\it fiber over $x$}. 
\end{notation}

\begin{warning} Unless $x$ is a rank-$1$ point, the underlying topological space
\[
|\Spa\left(K(x), K(x)^+\right)|
\]
is not just one point $\{x\}$. Instead, it is the set of all generalizations of $x$. In particular, the adic space $\bf{P}_x(\cal{E})$ is not literally the fiber over $x$ unless $x$ is of rank-$1$.
\end{warning}

\begin{lemma}\label{lemma:locally-constant-n} Let $S=\Spa(A, A^+)$ be a connected strongly noetherian Tate affinoid, let $f\colon \bf{P}^d_S \to S$ be the relative projective space of relative dimension $d\geq 1$, and let $\cal{N}$ be a line bundle on $\bf{P}^d_S$. Suppose that there is a point $x\in S$ such that \[
\cal{N}|_{\bf{P}^d_{x}} \simeq \O.
\]
Then 
\begin{enumerate}
    \item $f_*\cal{N}$ is a line bundle on $S$;
    \item the natural morphism $f^*f_*\cal{N} \to \cal{N}$ is an isomorphism.
\end{enumerate}
In particular, the restriction of $\cal{N}$ onto any fiber is trivial. 
\end{lemma}
\begin{proof}
    Using Lemma~\ref{lemma:GAGA} and the GAGA Theorem (see \cite[Thm.\,II.9.4.1]{FujKato}), we easily reduce the claim to an analogous claim for the algebraic relative projective space
    \[
    g\colon \bf{P}^{d, \rm{alg}}_A \to \Spec A.
    \]
    Then the results are well-known (for example, they can be deduced from \cite[Prop.~11.4.2]{Gabber-Ramero}) as long as we know that $\Spec A$ is connected. However, connectedness of $\Spec A$ follows from Lemma~\ref{lemma:connected-different}. 
\end{proof}

\begin{cor}\label{cor:locally-constant-n-general} Let $S$ be a locally noetherian analytic adic space, let $\cal{E}$ be a vector bundle on $S$ of rank at least $2$, let  $f\colon \bf{P}_S(\cal{E}) \to S$ be the associated projective bundle, and let $\cal{N}$ be a line bundle on $\bf{P}_S(\cal{E})$. For each integer $n$, let $E_n(\cal{N})$ be the set
\[
E_{n}(\cal{N}) \coloneqq \{ x\in S \ | \ \cal{N}|_{\bf{P}_x(\cal{E})} \simeq \O(n)\} \subset S.
\]
Then $E_{n}(\cal{N})$ is a clopen subset of $S$ for each integer $n$.
\end{cor}
\begin{proof}
    Since the subsets $E_{n}(\cal{N})$ are disjoint, it suffices to show that each of them is open. This follows directly from Lemma~\ref{lemma:locally-constant-n} and Lemma~\ref{lemma:locally-connected}. 
\end{proof}

\begin{lemma}\label{lemma:coherent-proj-formula} Let $S$ be a locally noetherian analytic adic space, let $\cal{E}$ be a vector bundle on $S$ of rank at least $2$, and let  $f\colon \bf{P}_S(\cal{E}) \to S$ be the associated projective bundle. Then, for any line bundle $\cal{L}\in \rm{Pic}(S)$, the natural morphism
\[
\cal{L} \to f_*f^*\cal{L}
\]
is an isomorphism.
\end{lemma}
\begin{proof}
    The proof is clearly local on $S$, so we can assume that $S$ is affine and both $\cal{L}$ and $\cal{E}$ are trivial. In this case, it suffices to show that the natural morphism
    \[
    \O_S \to f_*\O_{\bf{P}^d_S}
    \]
    is an isomorphism. This is standard and follows, for example, from the analogous algebraic results and the (relative) GAGA Theorem (see \cite[Thm.\,II.9.4.1]{FujKato}). 
\end{proof}

Now we are ready to give a proof of Theorem~\ref{thm:relative-picard-projective-line}.

\begin{proof}[Proof of Theorem~\ref{thm:relative-picard-projective-line}]
    {\it Step~$1$. Injectivity of $\alpha\colon \rm{Pic}\left(S\right) \bigoplus \Z \to \rm{Pic}\left(P\right)$.} Suppose that the map is not injetive, so there is a line bundle $\cal{N}=f^*\cal{L} \otimes \O_{P/S}(n)$ that is isomorphic to $\O_{P}$. Then Corollary~\ref{cor:case-of-a-field} implies that $n=0$ by restricting $\cal{N}$ onto the fiber over some rank-$1$ point $x\in S$. Thus 
    \[
    \O_P \simeq \cal{N}\simeq f^*\cal{L}.
    \] 
    In this case, Lemma~\ref{lemma:coherent-proj-formula} implies that
    \[
    \cal{L} \simeq f_*f^*\cal{L} \simeq f_*\O_P \simeq \O_S
    \]
    finishing the proof. \smallskip
    
    {\it Step~$2$. Surjectivity of $\alpha \colon \rm{Pic}\left(S\right) \bigoplus \Z \to \rm{Pic}\left(P\right)$.} Pick any object $\cal{N} \in \rm{Pic}(P)$ and a point $x\in S$. By Corollary~\ref{cor:case-of-a-field}, we know that $\cal{N}_x \simeq \O_{P_x}(n)$ for some integer $n$. Then Corollary~\ref{cor:locally-constant-n-general} implies that, for any point $y\in S$, 
    \[
    \cal{N}_y \simeq \O_{P_y}(n).
    \]
    Therefore, by replacing $\cal{N}$ with $\cal{N} \otimes \O_{P/S}(-n)$, we can assume that the restriction of $\cal{N}$ on any fiber is trivial. In this case, it suffices to show that 
    \[
    f_*\cal{N}
    \]
    is a line bundle on $S$, and that the natural morphism
    \[
    f^*f_*\cal{N} \to \cal{N}
    \]
    is an isomorphism. This question is local on $S$, so the result follows from Lemma~\ref{lemma:locally-constant-n} and Lemma~\ref{lemma:locally-connected}.
\end{proof}

\begin{cor}\label{cor:clopen-connected-components} Let $S$ be a locally noetherian analytic adic space, let $\cal{E}$ be a vector bundle on $S$ of rank at least $2$, let  $f\colon \bf{P}_S(\cal{E}) \to S$ be the associated projective bundle, and let $\cal{N}$ be a line bundle on $\bf{P}_S(\cal{E})$. Then there is a disjoint decomposition of $S$ into clopen subsets $S=\sqcup_{i\in I} S_i$ with the induced morphisms
\[
f_i\colon \bf{P}_{S_i}(\cal{E}|_{S_i}) \to S_i
\]
such that 
\[
\cal{N}|_{\bf{P}_{S_i}(\cal{E}|_{S_i})} \simeq f_i^* \cal{L}_i \otimes \O(n_i)
\]
for some $\cal{L}_i\in \rm{Pic}(S_i)$ and integers $n_i$.
\end{cor}
\begin{proof}
    This follows directly from Theorem~\ref{thm:relative-picard-projective-line} and Corollary~\ref{cor:finite-number-of-components}. 
\end{proof}

\section{\'Etale $6$-functor formalism}\label{section:6-functors-construction}

In this section, we construct an \'etale $6$-functor formalism on the category of locally noetherian analytic adic spaces. We refer to \cite[Appendix A.5]{Lucas-thesis} for the extensive discussion of $6$-functor formalisms and to \cite[Def.\,2.3.10 and Rem.\,2.3.11]{duality-revisited} for the precise definition of a $6$-functor formalism that we are going to use in these notes. Here, we only say that a data of a $6$-functor formalism is a formal way of encoding the $6$-functors
\[
\left(f^*, \, \rm{R}f_*, \, \otimes^L,\, \rm{R}\ud{\Hom},\, \rm{R}f_!,\, \rm{R}f^!\right)
\]
with all (including ``higher'') coherences between these functors. In particular, this encodes the projection formula and proper base-change. \smallskip

This formalism was essentially constructed by R.\,Huber in \cite{H3}. However, at some places he had to work with bounded derived categories and a restricted class of morphisms. We eliminate all these extra assumptions in this section, and also make everything $\infty$-categorical. \smallskip

In the rest of the section, we fix an integer $n$. For each locally noetherian analytic adic space $X$, we denote by $\cal{D}(X_\et; \Z/n\Z)$ the $\infty$-derived category of \'etale sheaves of $\Z/n\Z$-modules on $X$. This is a stable, presentable $\infty$-category with the standard $t$-structure (see \cite[Prop.\,1.3.5.9 and Prop.\,1.3.5.21]{HA}). \smallskip

We wish to define $6$-functors on $\cal{D}(X_\et; \Z/n\Z)$. We note that $4$-functors come for free. The category $\cal{D}(X_\et; \Z/n\Z)$ admits the natural symmetric monoidal structure by deriving the usual tensor product on $\rm{Shv}(X_\et; \Z/n\Z)$ (see \cite[Lemma 2.2.2 and Notation 2.2.3]{Liu-Zheng} for details). We denote this functor by
\[
- \otimes^L - \colon \cal{D}(X_\et; \Z/n\Z)\times \cal{D}(X_\et; \Z/n\Z) \to \cal{D}(X_\et; \Z/n\Z).
\]
By deriving the inner-Hom functor, we also get the functor
\[
\rm{R}\ud{\Hom}_X(-, -) \colon \cal{D}(X_\et; \Z/n\Z)^{\rm{op}} \times \cal{D}(X_\et; \Z/n\Z) \to \cal{D}(X_\et; \Z/n\Z)
\]
that is right adjoint to the tensor product functor. Similarly, for any morphism $f\colon X \to Y$ of locally noetherian analytic adic spaces, we get a pair of adjoint functors
\[
f^* \colon \cal{D}(Y_\et; \Z/n\Z) \to \cal{D}(X_\et; \Z/n\Z),
\]
\[
\rm{R}f_*\colon \cal{D}(X_\et;\Z/n\Z) \to \cal{D}(Y_\et; \Z/n\Z).
\]

Thus, the question of constructing $6$-functor essentially reduces to the question of constructing $\rm{R}f_!$ and $f^!$-functors and showing certain compatibilities.

\begin{lemma1}\label{lemma:properties-proper-morphisms} Let $f\colon X \to Y$ be a $+$-weakly finite type morphism of locally noetherian analytic adic spaces, and let $n>0$ be an integer. Then 
\begin{enumerate}
    \item\label{lemma:properties-proper-morphisms-1} if $Y$ is quasi-compact, $f_* \colon \rm{Shv}(X_\et; \Z/n\Z) \to \rm{Shv}(Y_\et; \Z/n\Z)$ is of finite cohomological dimension;
    \item\label{lemma:properties-proper-morphisms-2} $\rm{R}f_*$ commutes with all (homotopy) colimits. So it admits a right adjoint functor $f^\times \colon \cal{D}(Y_\et; \Z/n\Z) \to \cal{D}(X_\et; \Z/n\Z)$;\footnote{When $f$ is proper, the functor $f^\times$ coincides with the latter defined $\rR f^!$. However, these two functors differ in general.} 
    \item\label{lemma:properties-proper-morphisms-4} if $n$ is invertible in $\O_Y^+$, then, for any Cartesian diagram
    \[
    \begin{tikzcd}
    X' \arrow{r}{g'} \arrow{d}{f'} & X \arrow{d}{f} \\
    Y' \arrow{r}{g} & Y,
    \end{tikzcd}
    \]
    the natural morphism
    \[
    g^*\rm{R}f_* \to \rm{R}f'_* g'^*
    \]
    is an isomorphism of functors $\cal{D}(X_\et; \Z/n\Z) \to \cal{D}(Y'_\et; \Z/n\Z)$.
\end{enumerate}
\end{lemma1}
\begin{proof}    
    {\it Step~$1$. We show (\ref{lemma:properties-proper-morphisms-1}).} Since $f$ and $Y$ are quasi-compact, the claim is analytic local on both $X$ and $Y$. Therefore, we can assume that $X=\Spa(B, B^+)$ and $Y=\Spa(A, A^+)$ are both strongly noetherian Tate affinoids. In this case, we show that the $n$-cohomological dimension of $f_*$ is less than or equal to $2N$. For this, we consider the universal compactification
    \[
    \begin{tikzcd}
        X \arrow{d}{f} \arrow{r}{j} & \ov{X}^{/Y} \arrow{ld}{\ov{f}^{/Y}}\\
        Y, &  
    \end{tikzcd}
    \]
    where $j$ is a quasi-compact open immersion and $\ov{f}^{/Y}$ is proper (see \cite[Thm.~5.1.5 and Cor.~5.1.6]{H3}). Furthermore, \cite[Cor.~5.1.14]{H3} ensures that $\rm{dim.tr}(\ov{f}^{/Y})=\rm{dim.tr}(f)=N$. \smallskip

    Now \cite[Prop.~2.6.4]{H3} implies that $\rm{R}j_*=j_*$, so $\rm{R}f_* = \rm{R}\ov{f}^{/Y}_* \circ j_*$. Therefore, it suffices to show the $n$-cohomological dimension of $\ov{f}^{/Y}_*$ is less or equal to $2N$. This now follows directly from \cite[Prop.~5.3.11]{H3}. \smallskip
    
    {\it Step~$2$. We show (\ref{lemma:properties-proper-morphisms-2})}. Since $\rm{R}f_*$ is a right adjoint, it commutes with all finite limits. Therefore, it commutes with all finite colimits by \cite[Prop.\,1.1.4.1]{HA}, so it suffices to show that $\rm{R}f_*$ commutes with infinite direct sums. Therefore, it suffices to show that, for any collection of objects $\F_i\in \cal{D}(X_\et; \Z/n\Z)$, the natural morphism
    \[
    \psi\colon \bigoplus_{i\in I} \rm{R}f_* \left(\F_i \right) \to \rm{R}f_* \big( \bigoplus_{i\in I} \F_i \big)
    \]
    is an isomorphism. If all $\F_i \in \cal{D}^{\geq -c}(X; \Z/n\Z)$ for some fixed integer $c$, then this follows from \cite[Lemma 2.3.13(ii)]{H3}. In general, the claim is local on $Y$, so we can assume that $Y$ is quasi-compact. Then $\rm{R}f_*$ is of finite cohomological dimension by Step~$1$. Therefore, the unbounded version follows from the bounded one by a standard argument with truncations. We spell out this argument for the reader's convenience. 

    First, we note that it suffices to show that $\cH^i(\psi)$ is an isomorphism for all $i$. Since we are allowed to replace each $\F_i$ with $\F_i[i]$, we conclude that it suffices to show that $\cH^0(\psi)$ is an isomorphism. Now say $f_*$ has cohomological dimension $n$. Then \cite[\href{https://stacks.math.columbia.edu/tag/07K7}{Tag 07K7}]{stacks-project} implies that the morphisms
    \[
    \bigoplus_{i\in I} \rR f_* \F_i \to \bigoplus_{i\in I} \rR f_* \tau^{\geq -n} \F_i \text{ and }
    \]
    \[
    \rR f_* \big( \bigoplus_{i\in I} \F_i\big) \to \rR f_*\big( \bigoplus_{i\in I} \tau^{\geq -n}\F_i\big)
    \]
    are isomorphisms on $\cH^0$. Therefore, we conclude that, for the purpose of showing that $\cH^0(\psi)$ is an isomorphism, we can replace each $\F_i$ with $\tau^{\geq -n} \F_i$. But this case was already proven above. This finishes the proof that $\rR f_*$ commutes with all (homotopy) colimits. 
    
    The existence of a right adjoint follows directly from the fact that $\rm{R}f_*$ commutes with colimits and \cite[Cor.\,5.5.2.9]{HTT}. \smallskip
    
    {\it Step~$3$. We show (\ref{lemma:properties-proper-morphisms-4})}. The question is clearly analytically local on $Y$ and $Y'$, so we can assume that both spaces are quasi-compact. Therefore, Step~$2$ ensures that both $\rm{R}f_*$ and $\rm{R}f'_*$ have finite cohomological dimension. Therefore, a standard argument with truncations allows us to reduce to the case of bounded above complexes. In this case, we wish to show that the natural morphism 
    \[
    \psi_\F\colon g^*\rm{R}f_*\F \to \rm{R}f'_* \left( g'^*\F\right)
    \]
    is an isomorphism for any $\F\in \cal{D}^+(X_\et; \Z/n\Z)$. An easy argument with spectral sequences reduces the question to the case of a sheaf $\F\in \rm{Shv}(X_\et; \Z/n\Z)$. It suffices to show that $\psi_\F$ is an isomorphism on stalks at geometric points of $Y$. Then \cite[Lem.\,2.5.12 and Prop.\,2.6.1]{H3} reduce the question to the case of a surjective morphism 
    \[
    Y'=\Spa(C', C'^+) \to Y=\Spa(C, C^+)
    \]
    for some algebraically closed non-archimedean fields $C$ and $C'$ and open, bounded valuation subrings $C^+\subset C$ and $C'^+ \subset C'$. In this case, the result follows from \cite[Cor.\,4.3.2]{H3}. 
\end{proof}

Now we discuss the fifth functor $f_!$. The idea is to define it separately for an \'etale morphism and a proper morphism, and then show that these two functors ``glue\text{''} together.

\begin{lemma1}\label{lemma:base-change-etale} Let $j\colon U\to X$ be an \'etale morphism of locally noetherian analytic adic spaces, and let $n>0$ be an integer. Then the functor $j^*\colon \cal{D}(X_\et; \Z/n\Z) \to \cal{D}(U_\et; \Z/n\Z)$ admits a left adjoint 
\[
j_!\colon \cal{D}(U_\et; \Z/n\Z) \to \cal{D}(X_\et; \Z/n\Z)
\]
such that 
\begin{enumerate}
    \item for any Cartesian diagram
    \[
    \begin{tikzcd}
        U' \arrow{d}{j'} \arrow{r}{g'} & U \arrow{d}{j} \\
        X' \arrow{r}{g} & X
    \end{tikzcd}
    \]
    of locally noetherian analytic adic spaces, the natural morphism
    \[
    j'_! \circ (g')^* \to g^* \circ j_!
    \]
    is an isomorphism of functors $\cal{D}(U) \to \cal{D}(X')$;
    \item the natural morphism
    \[
    j_! (-\otimes^L j^*(-)) \to j_!(-) \otimes^L -
    \]
    is an isomorphism of functors $\cal{D}(U_\et; \Z/n\Z)\times \cal{D}(X_\et; \Z/n\Z) \to \cal{D}(X_\et; \Z/n\Z)$.
\end{enumerate}
\end{lemma1}
\begin{proof}
    {\it We first show existence of $j_!$.} For this, we note the \'etale topos $U_\et$ is the slice topos $(X_{\et})_{/h_U}$. Therefore, the pullback functor $j^*$ commutes with both limits and colimits. Since both $\cal{D}(U; \Z/n\Z)$ and $\cal{D}(X; \Z/n\Z)$ are presentable, the adjoint functor Theorem (see \cite[Cor.\,5.5.2.9]{HTT}) implies that $j^*$ admits a left adjoint $j_!$. \smallskip 
    
    
    {\it Base-change.} By adjunction, it suffices to show that the natural morphism 
    \[
    j^* \rm{R}g_* \to \rm{R}g'_* j'^*
    \]
    is an isomorphism of functors. This is essentially obvious because $U_\et$ is the slice topos of $X_\et$. \smallskip
    
    {\it Projection Formula.} This follows from Yoneda's Lemma and the following sequence of isomorphisms
    \begin{align*}
        \rm{Hom}_X\left(j_!\left(A\otimes^L j^*B\right), C\right) & \simeq \rm{Hom}_U\left(A\otimes^L j^*B, j^*C\right) \\
            & \simeq \rm{Hom}_U\left(A, \rm{R}\ud{\rm{Hom}}_U\left(j^*B, j^*C\right)\right) \\
            & \simeq \rm{Hom}_U\left(A, j^*\rm{R}\ud{\rm{Hom}}_X\left(B, C\right)\right) \\
            & \simeq \rm{Hom}_X\left(j_!A, \rm{R}\ud{\rm{Hom}}_X\left(B, C\right)\right) \\
            & \simeq \rm{Hom}_X\left(j_!A \otimes^L B, C\right). \qedhere
    \end{align*}
\end{proof}

Now we discuss the hardest part of the construction: we show that $j_!$ and $\rm{R}f_*$ are compatible in some precise sense:

\begin{prop1}\label{prop:projection-formula-proper} Let $Y$ be a locally noetherian analytic adic space, 
\[
\begin{tikzcd}
    X' \arrow{r}{j'} \arrow{d}{f'}& X \arrow{d}{f} \\
    Y' \arrow{r}{j} & Y
\end{tikzcd}
\]
a Cartesian diagram such that $f$ is proper and $j$ is \'etale. Then  
\begin{enumerate}
    \item there is a natural isomorphism of functors
    \[
    j_!\circ  \rm{R}f'_* \to \rm{R}f_* \circ j'_! \colon \cal{D}(X'_\et; \Z/n\Z) \to \cal{D}(Y_\et; \Z/n\Z).
    \]
    \item(Projection Formula) The natural morphism of functors 
    \[
    \rm{R}f_*(-)\otimes^L (-)\to \rm{R}f_*\big((-) \otimes^L f^*(-)\big) \colon \cal{D}(X) \times \cal{D}(Y) \to \cal{D}(Y)
    \]
    is an isomorphism.
\end{enumerate}
\end{prop1}
\begin{proof}
    {\it Part~(1).} Firstly, we define the morphism 
    \[
    \alpha \colon j_! \circ \rm{R}f'_* \to \rm{R}f_* \circ j'_!
    \]
    to be adjoint to the natural morphism
    \[
    f^* \circ j_! \circ \rm{R}f'_* \simeq j'_! \circ f'^* \circ \rm{R}f'_* \xr{j'_!(\rm{adj})} j'_!,
    \]
    where the first map comes from the base-change established in Lemma~\ref{lemma:base-change-etale}. The question whether $\alpha$ is an isomorphism is \'etale local on $Y$ and $Y'$, so we may assume that both spaces are affinoids. Then \cite[Lemma 2.2.8]{H3} ensures that, after possibly passing to an open covering of $Y$, there is a decomposition of $j$ into a composition $j=g\circ i$ such that $i$ is an open immersion and $g$ is a finite \'etale morphism. \smallskip
    
    It suffices to treat these two cases separately. Suppose first that $j$ is finite \'etale. We can check that $\alpha$ is an isomorphism \'etale locally on $Y$, so we can reduce to the case when $Y'$ is a disjoint union of copies of $Y$. Then the result is evident. \smallskip
    
    Now we deal with the case when $j$ is an open immersion. Since $\rm{R}f_*$ and $j_!$ both have finite cohomological dimesion, a standard argument reduces the question to showing that the natural morphism
    \[
    \alpha^i_\F \colon j_! \circ \rm{R}^if'_*\F \to \rm{R}^if_* \circ j'_! \F
    \]
    is an isomorphism for any $\F\in \rm{Shv}(X'_\et; \Z/n\Z)$ and $i\geq 0$. It suffices to verify this claim on stalks. Since both $\rm{R}f_*$ and $j_!$ commute with taking stalks, we can use \cite[Prop.~2.6.1]{H3} and Lemma~\ref{lemma:base-change-etale} to reduce the question to showing that the natural morphism
    \[
    \rm{R}\Gamma(X, j'_!\F) \to \rm{R}\Gamma(Y, j_!\rm{R}'_*\F)
    \]
    is an isomorphism, where $Y=\Spa(C, C^+)$ for an algebraically closed non-archimedean field $C$ and an open and bounded valuation subring $C^+\subset C$. If $Y' =Y$, then the claim is evident. Otherwise, we see that 
    \[
    \rm{R}\Gamma(Y, j_!\rm{R}'_*\F) = (j_!\rm{R}'_*\F)_{\ov{s}} = 0 
    \]
    since the stalk of $j_!\rm{R}'_*\F$ at the unique closed point $s\in \Spa(C, C^+)$ is zero. Therefore, it suffices to show that 
    \[
    \rm{R}\Gamma(X, j'_!\F)=0
    \]
    in this case. This is follows from \cite[Prop.\,4.4.3]{H3} since the restriction of $j'_!\F$ onto the fiber\footnote{This fiber is merely a pseudo-adic space, and not an adic space.} of $f$ over the closed point of $Y$ is equal to $0$. \smallskip
    
     {\it Part~(2).}  We wish to prove that the natural morphism
     \[
     \rm{R}f_*(\F\otimes^L f^*\G) \to \rm{R}f_*(\F) \otimes^L \G
     \]
     is an isomorphism for any $\F\in \cal{D}(X; \Z/n\Z)$ and $\G\in \cal{D}(Y;\Z/n\Z)$. Now we choose any complex $\G^\bullet$ representing $\G$. Then we note that the natural morphism
     \[
     \hocolim_N \sigma^{\geq -N} \G^\bullet \to \G
     \]
     is an isomorphism. Since all functors commute with (homotopy) colimits, it suffices to prove the result for $\sigma^{\geq -N} \G^\bullet$, i.e., we can assume that $\G\in \cal{D}^+(Y;\Z/n\Z)$. Then we may similarly use that 
     \[
     \hocolim_N \tau^{\leq N} \G\xr{\sim} \G
     \]
     to reduce to the case of a bounded complex $\G\in \cal{D}^b(Y;\Z/n\Z)$. This, in turn, can be reduced to the case when $\G\in \rm{Shv}(Y_\et; \Z/n\Z)$ by an easy induction on the number of non-zero cohomology sheaves. Then \cite[\href{https://stacks.math.columbia.edu/tag/0GLW}{Tag 0GLW}]{stacks-project} implies that $\G$ is a colimit of sheaves of the form $j_! \ud{\Z/n\Z}$ for some \'etale morphism $j\colon U \to X$. Again, since all functors in the question commute with all (homotopy) colimits, it suffices to prove the claim for $\G=j_! \ud{\Z/n\Z}$. In this case, this follows from the following sequence of isomorphisms
     \begin{align*}
     (\rm{R}f_* \F)\otimes^L j_!(\ud{\Z/n\Z}) & \simeq j_! j^* \rm{R}f_* \F \\
     & \simeq j_! \rm{R}f'_* (j')^*\F \\
     & \simeq \rm{R}f_*\circ j'_!\circ (j')^*\F \\
     & \simeq \rm{R}f_*(\F \otimes^L j'_! \ud{\Z/n\Z}) \\
     & \rm{R}f_*(\F \otimes f^* j_! \ud{\Z/n\Z}). \qedhere
     \end{align*}
\end{proof}

Before reading the proof of the next theorem, we strongly advise the reader to look at \cite[Appendix A.5]{Lucas-thesis} and \cite[\textsection 2.1, 2.3]{duality-revisited}. 

\begin{thm1}\label{thm:etale-six-functors} Let $S$ be a locally noetherian analytic adic space, let $\cal{C}'$ be the category of locally $+$-weakly finite type $S$-adic spaces, and let $n>0$ be an integer {\it invertible} in $\O_S^+$. Then there is a $6$-functor formalism (in the sense\footnote{Note that this definition slightly differs from \cite[Def.\,A.5.7]{Lucas-thesis}} of \cite[Def.\, 2.3.10 and Rmk.\,2.3.11]{duality-revisited})
\[
\cal{D}_{\et}(-; \Z/n\Z) \colon \rm{Corr}(\cal{C'}) \to \Cat_\infty
\]
such that 
\begin{enumerate}
    \item there is a canonical isomorphism of symmetric monoidal $\infty$-categories $\cal{D}_\et(X;\Z/n\Z) = \cal{D}(X_\et; \Z/n\Z)$ for any $X\in \cal{C'}$;
    \item for a  morphism $f\colon X \to Y$ in $\cal{C}'$, we have 
    \[
    \cal{D}_\et\big([Y \xleftarrow{f} X \xr{\rm{id}} X]\big) = f^*\colon \cal{D}(Y_\et; \Z/n\Z) \to \cal{D}(X_\et; \Z/n\Z);
    \]
    \item for a separated, taut, locally $+$-weakly finite type morphism $f\colon X \to Y$, we have
    \[
    \cal{D}_{\et}\big([X \xleftarrow{\rm{id}} X \xr{f} Y]\big)|_{\cal{D}_\et^+(X_\et; \Z/n\Z)} \simeq \rm{R}^+f_! \colon \cal{D}^+_\et(X_\et; \Z/n\Z) \to \cal{D}_\et(Y_\et; \Z/n\Z),
    \]
    where $\rm{R}^+f_!$ is the functor from \cite[Thm.~5.4.3]{H3}.
\end{enumerate}
\end{thm1}
\begin{proof}
    We use \cite[Lemma 2.2.2 and Notation 2.2.3]{Liu-Zheng} to get a functor 
    \[
    \cal{D}^*(-;\Z/n\Z) \colon \cal{C'}^{\rm{op}} \to \Cat_\infty^{\otimes}
    \]
    that sends a locally $+$-weakly finite type adic $S$-space $X$ to $\cal{D}(X_\et; \Z/n\Z)$. We extend it to the desired functor \[
    \cal{D}_\et(-; \Z/n\Z)\colon \rm{Corr}(\cal{C}')_{\rm{all}, \rm{all}} \to \Cat_\infty
    \]
    in four steps: \smallskip
    
    {\it Step~$1$. We define $\cal{D}_\et$ on ``compatifiable'' morphisms.} More precisely, we define $E\subset \rm{Hom}(\cal{C'})$ to be the class of $+$-weakly finite type, separated, taut morphisms (in the sense of \cite[Def.\,5.1.2]{H3}). We also define the subclasses 
    \[
    I, P \subset E
    \]
    to be quasi-compact open immersions and proper morphisms, respectively. Now \cite[Cor.\,5.1.6]{H3} implies that any morphism $f\in E$ admits a decomposition $f=p\circ i$ such that $i\in I$ and $p\in P$. One easily checks that $I, P \subset E$ defines a {\it suitable decomposition} of $E$ in the sense of \cite[Def.\,A.5.9]{Lucas-thesis}. Now Lemma~\ref{lemma:properties-proper-morphisms}, Lemma~\ref{lemma:base-change-etale}, and Proposition~\ref{prop:projection-formula-proper} ensure that all the conditions of \cite[Prop.\,A.5.10]{Lucas-thesis} are satisfied, and so it defines a weak $6$-functor formalism (see \cite[Def.\,2.1.2]{duality-revisited}) 
    \[
    \cal{D}_\et(-; \Z/n\Z) \colon \rm{Corr}(\cal{C}')_{E, \rm{all}} \to \Cat_\infty.
    \]
    Now recall that a $6$-functor formalism $\cal{D}_\et$ defines a lower-shriek functor $f_!$ for any morphism $f\in E$ (see \cite[Def.\,A.5.6]{Lucas-thesis}). In this case, the construction tells us that the lower shriek functor $f_!$ is equal to  $\rm{R}g_*\circ j_!$, where $f=g\circ j$ is the decomposition of $f$ into a composition of an open immersion $j$ and a proper morphism $g$. In particular, for a proper morphism $f$, we get an equality $f_!=\rm{R}f_*$. Thus, any proper morphism is cohomologically proper in the sense of \cite[Def.\,2.3.4]{duality-revisited}. \smallskip
    
    {\it Step~$2$. We extend $\cal{D}_\et$ to separated, locally $+$-weakly finite type morphisms.} We define $E_1$ to be the class of morphisms of the form $\sqcup_{i\in I} X_i \to Y$ such that each $X_i \to Y$ lies in $E$. Then \cite[Prop.\,A.5.12]{Lucas-thesis} ensures that $\cal{D}_\et(-; \Z/n\Z)$ uniquely extends to a weak $6$-functor formalism
    \[
    \cal{D}_\et(-; \Z/n\Z) \colon \rm{Corr}(\cal{C}')_{E_1, \rm{all}} \to \Cat_\infty.
    \]
    Now we define a new class of morphisms $E'_1$ to be the class of locally $+$-weakly finite type, separated morphism. We also define a subclass $S_1\subset E_1$ to consist of morphisms $\sqcup_{i\in I} U_i \to X$ for covers $X = \cup_{i\in I} U_i$ by quasi-compact open immersions. Then \cite[Prop.\,A.5.14]{Lucas-thesis} implies that $\cal{D}_\et(-; \Z/n\Z)$ uniquely extends to a weak $6$-functor formalism
    \[
    \cal{D}_\et(-; \Z/n\Z) \colon \rm{Corr}(\cal{C}')_{E'_1, \rm{all}} \to \Cat_\infty.
    \]
    
    {\it Step~$3$. We extend $\cal{D}_\et$ to all locally $+$-weakly finite type morphisms.} This reduction is similar to Step~$2$. We define $E''$ to be the collection of all locally $+$-weakly finite type morphisms, and $S\subset E'$ to be the collection of morphisms $\sqcup_{i\in I} U_i \to X$ for covers $X = \cup_{i\in I} U_i$ by open immersions. Then \cite[Prop.\,A.5.12]{Lucas-thesis} and \cite[Prop.\,A.5.14]{Lucas-thesis} imply that $\cal{D}_\et(-; \Z/n\Z)$ uniquely extends to a weak $6$-functor formalism
    \[
    \cal{D}_\et(-; \Z/n\Z) \colon \rm{Corr}(\cal{C}')_{\rm{all}, \rm{all}} \to \Cat_\infty.
    \]
    
    {\it Step~$4$. We show that $\cal{D}_\et$ is a $6$-functor formalism in the sense of \cite[Def.~2.3.10 and Rmk.\,2.3.11]{duality-revisited}.} We already have a weak $6$-functor formalism
    \[
        \cal{D}_\et(-; \Z/n\Z) \colon \rm{Corr}(\cal{C}')_{\rm{all}, \rm{all}} \to \Cat_\infty.
    \]
    By construction, the categories $\cal{D}_\et(X; \Z/n\Z) \simeq \cal{D}(X_\et;\Z/n\Z)$, so they are stable and presentable. Clearly, $\cal{D}_\et$ satisfies analytic descent; it even satisfies \'etale descent. By Step~$1$, we know that any proper morphism $f\colon X \to Y$ is cohomologically proper (in the sense of \cite[Def.\,2.3.4]{duality-revisited}). Therefore, we are only left to check that any \'etale morphism $j\colon X \to Y$ is cohomologically \'etale in the sense of \cite[Def.\,2.3.4]{duality-revisited}. \smallskip
    
    For this, we set up $E=\et$ to be the class of all \'etale morphisms, and restrict $\cal{D}_\et$ onto $\Corr(\cal{C}')_{\et, \rm{all}}$ to get a weak $6$-functor formalism
    \[
    \cal{D}'_\et \colon \Corr(\cal{C}')_{\et, \rm{all}} \to \Cat_\infty.
    \]
    Alternatively, we can can apply \cite[Prop.\,A.5.10]{Lucas-thesis} to $E=I$ being the class of all \'etale morphisms and the class $P$ consisting only of the identity morphisms to get another weak $6$-functor formalism
    \[
    \cal{D}''_\et \colon \Corr(\cal{C}')_{\et, \rm{all}} \to \Cat_\infty.
    \]
    By construction, any \'etale morphism is cohomologically \'etale with respect to $\cal{D}''_\et$. Thus, the question boils down to showing that $\cal{D}'_\et$ and $\cal{D}''_\et$ coincide. Using the uniqueness statements from  \cite[Prop.\,A.5.12, A.5.14, A.5.16]{Lucas-thesis}, we can repeat the same arguments as in Steps~$2$ and $3$ to reduce the question to showing that the restrictions
    \[
    \cal{D}'_\et|_{\text{\'etqcsep}} \colon \Corr_{\text{\'etqcsep}, \rm{all}} \to \Cat_\infty,
    \]
    \[
    \cal{D}''_\et|_{\text{\'etqcsep}} \colon \Corr_{\text{\'etqcsep}, \rm{all}} \to \Cat_\infty
    \]
    coincide, where $\text{\'etqcsep}$ stands for the class of \'etale quasi-compact, separated morphisms. Now we note that \'etale quasi-compact, separated morphisms are taut by \cite[Lemma 5.1.3(iv)]{H3}. Therefore, after unravelling the definitions, we see that both $\cal{D}'_\et|_{\text{\'etqcsep}}$ and $\cal{D}''_\et|_{\text{\'etqcsep}}$ are obtained by applying \cite[Prop.\,A.5.8]{Lucas-thesis} to $I={\text{\'etqcsep}}$ and $P={\rm{id}}$. Therefore, they coincide. 

    {\it Step~$5$. Compare to Huber's theory.} First, we have $\cal{D}_\et(X;\Z/n\Z) = \cal{D}(X_\et; \Z/n\Z)$ and $\cal{D}_\et\big([Y \xleftarrow{f} X \xr{\rm{id}} X]\big) = f^*$ by the very construction of $\cal{D}_\et$. Now we wish to show that
    \[
        \cal{D}_{\et}\big([X \xleftarrow{\rm{id}} X \xr{f} Y]\big)|_{\cal{D}_\et^+(X_\et; \Z/n\Z)} \simeq \rm{R}^+f_! \colon \cal{D}_\et(X_\et; \Z/n\Z) \to \cal{D}_\et(Y_\et; \Z/n\Z),
    \]
    for a separated, locally $+$-weakly finite type, taut morphism $f\colon X \to Y$. For brevity, we denote $\cal{D}_\et([X \xleftarrow{\rm{id}} X \xr{j} Y])$ by $\rm{R}_{\et}f_!$ and its restriction on $\cal{D}_\et^+(X_\et; \Z/n\Z)$ by $\rm{R}^+_{\et} f_!$. \smallskip

    Now we note that the construction of $\cal{D}_\et$ implies that $\rm{R}_{\et}j_! = j_!$ if $f=j\colon X \to Y$ is an \'etale morphism. Likewise, $\rm{R}_\et f_! = \rm{R}f_*$ if $f$ is proper. Now we use \cite[Thm.~5.1.5 and Cor.~5.1.6]{H3} and the fact that both $\rm{R}_\et f_!$ and $\rm{R}^+f_!$ are compatible with compositions to conclude that $\rm{R}_\et f_! = \rm{R}^+f_!$ for a separated, {\it $+$-weakly finite type}, taut morphism $f$. \smallskip
    
    Now using the case of \'etale morphisms and \cite[Thm.~5.1.5]{H3} again, we see that it suffices to show that $\rm{R}_\et^+f_! = \rm{R}^+f_!$ for partially proper morphism $f$. \smallskip
    
    In what follows, we assume that $f$ is partially proper. Then we note that \cite[Def.~5.3.1]{H3} says that $\rm{R}^+f_!$ is the right derived functor of the functor $f_!$ from \cite[Def.~5.2.1]{H3}. Now the dual versions of \cite[Thm.~1.3.3.2 and Ex.~1.3.3.4]{HA} imply that, in order to construct an equivalence of functors
    \[
    \rm{R}^+f_! \xrightarrow{\sim} \rm{R}^+_{\et} f_! \colon \cal{D}^+(X_\et; \Z/n\Z) \to \cal{D}(Y_\et; \Z/n\Z),
    \]
    it suffices to construct the following equivalence of functors
    \[
    f_! \xrightarrow{\sim} \cal{H}^0\big(\rm{R}^+_{\et} f_!\big) \colon \rm{Shv}(X_\et;\Z/n\Z) \to \rm{Shv}(Y_\et; \Z/n\Z).
    \]
    This can be done locally on $Y$, so we may and do assume that $Y$ is affinoid. In this case, we note that the proof of \cite[Prop.~5.2.2]{H3} ensures $f_! = \colim_{U\in \cal{F}} (f|_U)_! (-|_U)$, where $\cal{F}$ is the filtered system of quasi-compact open subsets $U\subset X$. \smallskip

    After unravelling the construction of $\rm{R}_\et f_!$, we get that 
    \[
    \rm{R}_\et f_! \simeq \colim_{U\in \cal{F}} (\rm{R}_\et f|_U)_! (-|_U).
    \] 
    Since filtered colimits are exact, we combine the above formulas with the established above case of separated, weakly $+$-finite type, taut morphisms to conclude that 
    \[
    \cal{H}^0(\rm{R}_{\et} f_!) \simeq \colim_{U\in \F} \cal{H}^0(\rm{R}_{\et} f|_U)_! (-|_U) \simeq \colim_{U\in \F} (f|_U)_! (-|_U) \simeq f_!.
    \]
    This finishes the proof. 
\end{proof}
    
\begin{rmk1}\label{rmk:restrict-6-functors} Let $S$ be a locally noetherian analytic adic space, $\cal{C}$ the category of locally finite type adic $S$-spaces, and $n$ is an integer invertible in $\O_S^+$. Then we can restrict the functor $\cal{D}_\et(-; \Z/n\Z)\colon \rm{Corr}(\cal{C}') \to \Cat_\infty$ onto $\rm{Corr}(\cal{C})$ to get the \'etale $6$-functor formalism
\[
\cal{D}_\et(-; \Z/n\Z) \colon \rm{Corr}(\cal{C}) \to \Cat_\infty.
\]
\end{rmk1}

\begin{rmk1}\label{rmk:algebraic-6-functors} Let $S$ is a scheme, $\cal{C}$ the category of locally finitely presented $S$-schemes, and $n$ any integer. Then one can similarly construct the \'etale $6$-functor formalism
\[
\cal{D}_\et(-; \Z/n\Z) \colon \rm{Corr}(\cal{C}) \to \Cat_\infty.
\]
 The proof of Theorem~\ref{thm:etale-six-functors} applies essentially verbatim. The main non-trivial input needed is:
 \begin{enumerate}
     \item (\cite[Thm.\,4.1]{Conrad2007}) Nagata's compatification;
     \item (\cite[Prop.\,5.9.6]{Lei-Fu}) the natural morphism 
     \[
     \bigoplus_I \rm{R}f_* \F_i \to  \rm{R}f_*\big(\bigoplus_I \F_i\big)
     \]
     for a proper morphism $f$ and a collection of sheaves $\{\F_i\in \rm{Shv}(X_\et; \Z/n\Z)\}_{i\in I}$;
     \item (\cite[Thm\,7.3.1]{Lei-Fu}) proper base-change for bounded below complexes;
     \item projection formula for proper $f$ and bounded below complexes (in this case, it follows automatically from (2) and (3) by arguing on stalks, see \cite[7.4.7]{Lei-Fu});
     \item finite cohomological dimension of $f_*$ for a proper $f$ (one can either adapt the proof of \cite[Thm.\,7.4.5]{Lei-Fu}\footnote{For this, one notices that (1) and (2) already imply that $\rm{R}f_!$ is a well-defined functor} or \cite[Cor.\,7.5.6]{Lei-Fu}).
 \end{enumerate}
    See also \cite[Appendix to Lecture VII]{scholze-notes} for a related discussion (with a slightly different $\cal{C}$). 
\end{rmk1}

\section{Overconvergent sheaves}\label{section:overconvergent}

In this section, we prove two basic facts about overconvergent sheaves. Both facts can be deduced from the results in \cite{H3}. However, the proofs in \cite{H3} seem to be unnecessary difficult, so we prefer to include alternative proofs of these facts in these notes.

\begin{defn}(\cite[Def.\,8.2.1]{H3}) An \'etale sheaf $\F \in \rm{Shv}(X_\et; \Z/n\Z)$ on a locally noetherian analytic adic space is  {\it overconvergent} if for every specialization of geometric points $u\colon \ov{\eta}  \to \ov{s}$, the specialization morphism
\[
\F_{\ov{s}} \to \F_{\ov{\eta}}
\]
is an isomorphism.
\end{defn}

First, we give the following basic example of overconvergent sheaves:

\begin{lemma}\label{lemma:algebraic-overconvergent} Let $S=\Spa(A, A^+)$ be a strongly noetherian Tate affinoid, let $X$ be a finite type $A$-scheme, let $c_{X/S}\colon X^{\an/S} \to X$ be the relative analytification morphism, and let $\F$ be an \'etale sheaf on $X$. Then $c_{X/S}^*\F$ is overconvergent.
\end{lemma}
\begin{proof}
    It suffices to show that, for any algebraically closed non-archimedean field $C$ with an open bounded valuation ring $C^+\subset C$ and a morphism $s\colon Y=\Spa(C, C^+) \to X^{\an/S}$, the pullback $s^*c_{X/S}^*\F$ is a constant sheaf. For this, we consider the following commutative diagram
    \[
        \begin{tikzcd}
            Y=\Spa(C, C^+) \arrow{d}{s} \arrow{r}{c_{Y}} & \Spec C \arrow{d}{s^{\rm{alg}}} \\
            X^{\an/S} \arrow{r}{c_{X/S}} & (X, \O_X).
        \end{tikzcd}
    \]
    This implies that $s^*c_{X/S}^*\F \simeq c_Y^*s^{\rm{alg}, *} \F$. Now the result follows from the observation that any sheaf on $(\Spec C)_\et$ is constant since $C$ is algebraically closed.
\end{proof}

\begin{lemma}\label{lemma:overconvergent} Let $Y$ be a locally noetherian analytic adic space, let $j\colon X \to Y$ be a quasi-compact dense pro-open immersion, let $n$ an integer, and let $\F\in \rm{Shv}(Y_\et; \Z/n\Z)$ be an overconvergent \'etale sheaf on $Y$. Then the natural morphism $\F\to \rm{R}j_*j^*\F$ is an isomorphism. In particular, the natural morphism
\[
\rm{R}\Gamma(Y, \F) \to \rm{R}\Gamma(X, j^*\F)
\]
is an isomorphism. 
\end{lemma}
\begin{proof}
    First, we note that we can check that 
    \begin{equation}\label{eqn:overconvergent}
        \F \to \rm{R}j_*j^*\F
    \end{equation} 
    is an isomorphism at the geometric points of $Y$. Therefore, we can assume that $j$ is of the form $j\colon \Spa(C, C'^+) \to \Spa(C, C^+)$ for an algebraically closed non-archimedean field $C$ and open and bounded valuation subrings $C'^+\subset C^+\subset C$. In this case, it suffices to show that 
    \[
    \rm{H}^i(\Spa(C, C'^+), j^*\F)=0
    \]
    for $i\geq 1$, and 
    \[
    \rm{H}^0(\Spa(C, C^+), \F) \to \rm{H}^0(\Spa(C, C'^+), j^*\F)
    \]
    is an isomorphism. The first follows from the fact that any {\it surjective} \'etale morphism $S \to \Spa(C, C'^+)$ has a section\footnote{First reduce to an affinoid $S$, then use \cite[Lemma 2.2.8]{H3} and an equivalence $\Spa(C, C'^+)_{\fet}\simeq (\Spec C)_{\fet}$ to construct a section.}. \smallskip
    
    Now we show the second claim. Let $\ov{s}=\rm{id}\colon \Spa(C, C^+) \to \Spa(C, C^+)$ be the geometric point of $\Spa(C, C^+)$ corresponding to its closed point, and let $\ov{s}'=j\colon \Spa(C, C'^+) \to \Spa(C, C^+)$ be the geometric point corresponding to the closed point of $\Spa(C, C'^+)$. Then we have $\rm{H}^0(\Spa(C, C^+), \F)=\F_s$ and $\rm{H}^0(\Spa(C, C'^+), j^*\F)\simeq \F_{s'}$. So the overconvergent assumption implies that the natural morphism
    \[
    \F_{\ov{s}} \to \F_{\ov{s}'}
    \]
    is an isomorphism finishing the proof. 
\end{proof}

Finally, we can show that overconvergent sheaves are closed under higher derived pushforwards along finite type, quasi-separated morphisms. In combination with Lemma~\ref{lemma:algebraic-overconvergent}, this allows to produce new interesting examples of overconvergent sheaves in analytic geometry. 

\begin{lemma}\label{lemma:overconvergent-preservation} Let $f\colon X \to S$ be a finite type, quasi-separated morphism of locally noetherian analytic adic spaces, $n$ an integer, and $\F\in \rm{Shv}(X; \Z/n\Z)$ an overconvergent sheaf. Then $\rm{R}^if_*\F$ is overconvergent for any $i\geq 0$.
\end{lemma}
\begin{proof}
    By \cite[Prop.\,2.6.1]{H3}, it suffices to show that, for any algebrically closed non-archimedean field $C$ with an open bounded valuation ring $C^+$ and a morphism $\Spa(C, C^+) \to S$, the natural morphism
    \[
    \rm{H}^i(X_{\Spa(C, C^+)}, \F) \to \rm{H}^i(X_{\Spa(C, \O_C)}, \F)
    \]
    is an isomorphism. This follows from Lemma~\ref{lemma:overconvergent}.
\end{proof}

\section{Categorical properties of lisse and constructible sheaves}\label{section:lisse}

In this section, we show that lisse and constructible \'etale sheaves on a locally noetherian analytic adic space (resp. a scheme) $X$ admit a nice categorical description. The results of this section are well-known to the experts, but it seems hard to find them explicitly stated in the existing literature. \smallskip

For the rest of the section, we fix a locally noetherian analytic adic space (resp. a scheme) $X$ and an integer $n>0$. \smallskip

We recall that the derived category $\cal{D}(X_\et; \Z/n\Z)$ admits a natural structure of a symmetric monoidal category (with the monoidal structure given by $-\otimes^L -$). In particular, it there is a well-defined notion of dualizable objects in $\cal{D}(X_\et; \Z/n\Z)$, see \cite[\href{https://stacks.math.columbia.edu/tag/0FFP}{Tag 0FFP}]{stacks-project}. 

\begin{lemma}\label{lemma:lisse-categorical} Let $X$ be a locally noetherian analytic adic space or a scheme, and let $n>0$ be an integer. Then the following are equivalent:
\begin{enumerate}
    \item $\F\in \cal{D}(X_\et; \Z/n\Z)$ is dualizable;
    \item $\F\in \cal{D}(X_\et; \Z/n\Z)$ is perfect;
    \item $\F$ lies in $\cal{D}^{(b)}_{\rm{lisse}}(X_\et; \Z/n\Z)$ and, for each geometric point $\ov{s} \to X$, the stalk $\F_{\ov{s}}$ is a perfect complex in $\cal{D}(\Z/n\Z)$.
\end{enumerate}
\end{lemma}
\begin{proof}
    First, we note that \cite[\href{https://stacks.math.columbia.edu/tag/0FPV}{Tag 0FPV}]{stacks-project} ensures that $\F$ is dualizable if and only if $\F$ is perfect.\smallskip
    
    Now suppose that $\F$ is a perfect complex. Then clearly all the stalks $\F_{\ov{s}}$ are perfect objects of $\cal{D}(\Z/n\Z)$. Furthermore, the definition of a perfect complex (see \cite[\href{https://stacks.math.columbia.edu/tag/08G5}{Tag 08G5}]{stacks-project}), \cite[\href{https://stacks.math.columbia.edu/tag/08G9}{Tag 08G9}]{stacks-project}, and the fact that lisse sheaves form a weak Serre subcategory of $\rm{Shv}(X_\et; \Z/n\Z)$ imply that the object $\F$ is locally bounded and has lisse cohomology sheaves, i.e. $\F\in \cal{D}^{(b)}_{\rm{lisse}}(X_\et; \Z/n\Z)$. \smallskip
    
    Now we suppose that $\F$ lies in $\cal{D}^{(b)}_{\rm{lisse}}(X_\et; \Z/n\Z)$ and all its stalks are perfects. We wish to show that $\F$ is perfect. This is a local question, so we can assume that $X$ is qcqs, and thus $\F$ lies in $\cal{D}^b_{\rm{lisse}}(X_\et; \Z/n\Z)$. Now \cite[\href{https://stacks.math.columbia.edu/tag/094G}{Tag 094G}]{stacks-project} implies that there is a (finite) covering $\{U_i\}_{i\in I} \to X$ such that 
    \[
    \F|_{U_i} \simeq \ud{M_i^\bullet}
    \]
    for some finite complexes of finite $\Z/n\Z$-modules $M_i^\bullet$. Using that all stalks of $\F$ are perfect as objects of $\cal{D}(\Z/n\Z)$, we conclude that each $M_i^\bullet$ must be perfect. This implies that $\F$ is a perfect object of $\cal{D}(X_\et; \Z/n\Z)$.
\end{proof}

Now we discuss the categorical description of constructible sheaves:\footnote{We refer to \cite[\textsection 2.7]{H3} and \cite[\href{https://stacks.math.columbia.edu/tag/05BE}{Tag 05BE}]{stacks-project} for the definition of constructible sheaves in the adic and schematic setups respectively.}

\begin{lemma}\label{lemma:abstract-nonsense} Let $X$ be a qcqs noetherian analytic adic space or a qcqs scheme, and let $n>0$ and $N$ be some integers. Then  an object $\F\in \cal{D}^{\geq -N}(X_\et; \Z/n\Z)$ is compact if and only if $\F$ lies in $\cal{D}^{b, \geq -N}_{\rm{cons}}(X_\et; \Z/n\Z)$, i.e., $\F$ is bounded and all its cohomology sheaves are constructible.
\end{lemma}
\begin{proof}
    Without loss of generality, we can assume that $N=0$. \smallskip 
    
    {\it Step~$1$. The ``if'' direction.} An easy argument using the Ext spectral sequence (see \cite[\href{https://stacks.math.columbia.edu/tag/07AA}{Tag 07AA}]{stacks-project}) implies that we can assume that $\F$ is an (abelian) constructible sheaf. Then the question boils down to showing that $\rm{R}\Gamma_\et(X, -)$ and $\rm{R}\ud{\rm{Hom}}^i(\F, -)$ commute with arbitrary direct sums in $\rm{Shv}(X_\et; \Z/n\Z)$. \smallskip

    First, we observe that $\rm{R}\Gamma_\et(X, -)$ commutes with direct sums (in $\rm{Shv}(X_\et; \Z/n\Z)$) due to \cite[Lemma 2.3.13(i)]{H3} and \cite[\href{https://stacks.math.columbia.edu/tag/03Q5}{Tag 03Q5}]{stacks-project}. \smallskip

    Now we show that $\rm{R}\ud{\rm{Hom}}^i(\F, -)$ commutes with direct sums (in $\rm{Shv}(X_\et; \Z/n\Z)$). For this, we first consider the case $\F=f_!\left(\ud{\Z/n\Z}\right)$ for a qcqs \'etale morphism $f\colon U \to X$. Then the claim follows from the isomorphism
    \[
    \rm{R}\ud{\Hom}_X(f_! \ud{\Z/n\Z}, -) \simeq \rm{R}f_*\rm{R}\ud{\Hom}_U(\ud{\Z/n\Z}, f^*-) \simeq \rm{R}f_*f^*(-)
    \]
    and the fact that $\rm{R}f_*$ commutes with direct sums (see \cite[Lemma 2.3.13(ii)]{H3} and \cite[\href{https://stacks.math.columbia.edu/tag/09Z1}{Tag 09Z1}]{stacks-project}). \smallskip

    Now, for a general constructible sheaf $\F$, we use a resolution of the form
    \[
    \dots \to f_{1, !} \ud{\Z/n \Z} \to f_{0, !} \ud{\Z/n\Z} \to \F \to 0,
    \]
    with $f_i\colon X_i \to X$ being qcqs \'etale maps (existence of such a presentation follows from \cite[\href{https://stacks.math.columbia.edu/tag/095N}{Tag 095N}]{stacks-project} in the scheme case and from the proof of \cite[Lemma 2.7.8]{H3} in the adic case). Then an easy argument with the Ext spectral sequence (see \cite[\href{https://stacks.math.columbia.edu/tag/07AA}{Tag 07AA}]{stacks-project}) implies that $\F$ is compact since each $f_{n, !}\left(\ud{\Z/n\Z}\right)$ is so. \smallskip

    {\it Step~$2$. We show that the natural morphism $\rm{Ind}\big(\cal{D}^{b, \geq 0}_{\rm{cons}}(X_\et; \Z/n\Z)\big) \to \cal{D}^{\geq 0}(X_\et; \Z/n\Z)$ is an equivalence.} First, we note that $\cal{D}^{\geq 0}(X_\et; \Z/n\Z)$ admits all (small) filtered colimits (see \cite[\href{https://kerodon.net/tag/03Y1}{Tag 03Y1}]{kerodon}), so the natural inclusion
    \[
        g\colon \cal{D}^{b, \geq 0}_{\rm{cons}}(X_\et; \Z/n\Z) \to \cal{D}^{\geq 0}(X_\et; \Z/n\Z)
    \]
    extends to the functor
    \[
        G\colon \rm{Ind}\big(\cal{D}^{b, \geq 0}_{\rm{cons}}(X_\et; \Z/n\Z)\big) \to \cal{D}^{\geq 0}(X_\et; \Z/n\Z)
    \]
    due to \cite[Lemma 5.3.5.8]{HTT}. Now \cite[Prop.~5.3.5.11(1)]{HTT} ensures that $G$ is fully faithful. So we are only left to show that $G$ is essentially surjective. \smallskip

    Now we note that the functor $G$ preserves finite direct sums because the same holds for $g$. Since $G$ preserves all filtered colimits, we conclude that $G$ preserves arbitrary direct sums. Furthermore,  \cite[Prop.~5.3.5.15]{HTT} implies that $G$ also preserves pushouts because $g$ does the same. Therefore, (the dual form of) \cite[\href{https://kerodon.net/tag/03UL}{Tag 03UL}]{kerodon} and \cite[\href{https://kerodon.net/tag/03UM}{Tag 03UM}]{kerodon} imply that $G$ commutes with all colimits. Therefore, it suffices to show that $\cal{D}^{b, \geq 0}_{\rm{cons}}(X_\et; \Z/n\Z)$ genenerates $\cal{D}^{\geq 0}(X_\et; \Z/n\Z)$ under colimits. \smallskip
    
    For this, we note that any object $\F\in \cal{D}^{\geq 0}(X_\et;\Z/n\Z)$ can be written as a (homotopy) colimit
    \[
    \colim_{n} \tau^{\leq -n}\F \to \F.
    \]
    Therefore, it suffices to show that $\cal{D}^{b, \geq 0}(X_\et; \Z/n\Z)$ is generated by $\cal{D}^{b, \geq 0}_{\rm{cons}}(X_\et; \Z/n\Z)$ under colimits. Since cofibers are colimits, we now reduce the question to showing that any abelian sheaf $\F\in \rm{Shv}\left(X_\et; \Z/n\Z\right)$ can be written as a (filtered) colimit of constructible abelian sheaves of $\Z/n\Z$-modules. This follows from \cite[Lemma 2.7.8]{H3} in the adic world and from \cite[\href{https://stacks.math.columbia.edu/tag/03SA}{Tag 03SA}]{stacks-project} in the scheme world. \smallskip
    
    {\it Step~$3$. Finish the proof.} Step~$2$ implies that any $\F\in \cal{D}^{\geq 0}(X_\et; \Z/n\Z)$ can be written as a filtered (homotopy) colimit
    \[
    \F = \colim_{i\in I} \F_i
    \]
    with $\F_i\in \cal{D}^{b, \geq 0}_{\rm{cons}}(X_\et; \Z/n\Z)$. If $\F$ is compact, we see that there is an equivalence
    \[
    \rm{Hom}(\F, \F) = \rm{Hom}(\F, \colim_{i\in I} \F_i) = \colim_{i\in I} \rm{Hom}(\F, \F_i).
    \]
    In particular, we note that the identity morphism $\rm{id} \colon \F \to \F$ factors through some $\F_i \to \F$. Thus, $\F$ is a direct summand of $\F_i$, so it must lie in $\cal{D}^{b, \geq 0}_{\rm{cons}}(X_\et; \Z/n\Z)$. 
\end{proof}

\bibliography{biblio}

\end{document}